\documentclass{article}

\usepackage[utf8]{inputenc}

\usepackage{geometry}
\usepackage{microtype}

\usepackage{natbib}

\usepackage{microtype}
\usepackage[ruled]{algorithm2e}
\SetKwInput{Input}{Input}
\SetKwInput{Output}{Output}
\usepackage{amsmath}
\usepackage{amsthm}
\usepackage{lscape}
\usepackage{amssymb}
\usepackage{bbm}
\usepackage{mathrsfs}
\usepackage{tikz}
\usepackage{tikz-cd}
\usetikzlibrary{arrows}

\usepackage{graphicx}
\graphicspath{{figs/}}

\DeclareMathOperator{\im}{Im}
\DeclareMathOperator{\res}{res}
\DeclareMathOperator{\diam}{diam}

\def \St{\text{St}}
\def\R{\mathbbm{R}}
\def\P{\mathbbm{P}}
\def \X{\mathbb{X}}
\def \Y{\mathbb{Y}}
\def \F{\mathbb{F}}
\def\G{\mathbb{G}}
\def \V{\mathbb{V}}
\def \E{\mathbb{E}}
\def \sR{\mathscr{R}}

\def \sC{\mathscr{C}}
\def \sG{\mathscr{G}}
\def \sD{\mathscr{D}}

\def \sF{\mathscr{F}}
\def \cC{{\cal C}}
\def\cR{{\cal R}}
\def \cD{{\cal D}}
\def \cS{{\cal S}}

\def \cI{{\cal I}}
\def \cT{{\cal T}}

\def\cM{{\cal M}}
\def\cN{{\cal N}}
\def\cV{{\cal V}}
\def \cU{{\cal U}}
\def \cW{{\cal W}}

\def\limninf{\lim_{n\to\infty}}

\newcommand{\prob}[1]{\P\left(#1\right)}

\def \Set {{\mathbf{Set}}}
\def \Int {{\mathbf{Int}}}

\def \Csh{{\mathbf{Csh}}}
\def \Cshc{{\mathbf{Csh^c}}}

\def \Rgraph{{\R\text{-}\mathbf{graph}}}
\def \Rspace{{\R\text{-}\mathbf{space}}}
\def \Rspacec{{\R\text{-}\mathbf{space^c}}}

\def \inv{^{-1}}

\renewcommand{\epsilon}{\varepsilon}
\newcommand{\e}{\varepsilon}
\renewcommand{\phi}{\varphi}

\newtheorem{theorem}{Theorem}
\newtheorem*{theorem*}{Theorem}
\numberwithin{theorem}{section}
\newtheorem{corollary}[theorem]{Corollary}
\newtheorem{lemma}[theorem]{Lemma}
\newtheorem{definition}[theorem]{Definition}
\newtheorem{example}[theorem]{Example}
\newtheorem{remark}[theorem]{Remark}
\newtheorem{proposition}[theorem]{Proposition}
\newtheorem*{acknowledgements}{Acknowledgements}

\newcommand {\mm}[1] {\ifmmode{#1}\else{\mbox{\(#1\)}}\fi}

\newcommand{\para}[1]        {\vspace{2mm}\noindent{\textbf{#1}}}
\newcommand{\denselist}{\vspace{-5pt} \itemsep -2pt\parsep=-1pt\partopsep -2pt}
\newcommand{\myblue}[1]{{\textcolor{black}{#1}}}

\begin{document}

\title{Probabilistic Convergence and Stability of Random Mapper Graphs}

\author{Adam Brown\thanks{E-mail: adam.brown@ist.ac.at} \\ IST Austria
\and Omer Bobrowski \thanks{E-mail: omer@ee.technion.ac.il}\\ Technion - Israel Institute of Technology
\and Elizabeth Munch \thanks{E-mail: muncheli@msu.edu}\\ Michigan State University
\and Bei Wang\thanks{E-mail: beiwang@sci.utah.edu} \\ University of Utah}
\date{}

\maketitle

\begin{abstract}

We study the probabilistic convergence between the mapper graph and the Reeb graph of a topological space $\X$ equipped with a continuous function $f: \X \to \R$. 
We first give a categorification of the mapper graph and the Reeb graph  by interpreting them in terms of cosheaves and stratified covers of the real line $\R$. We then introduce a variant of the classic mapper graph of Singh et al.~(2007), referred to as the enhanced mapper graph, and demonstrate that such a construction approximates the Reeb graph of $(\X, f)$ when it is applied to points randomly sampled from a probability density function concentrated on $(\X, f)$. 

Our techniques are based on the interleaving distance of constructible cosheaves and topological estimation via kernel density estimates.
Following Munch and Wang (2018), we first show that the mapper graph of $(\X, f)$, a constructible $\R$-space (with a fixed open cover),  approximates the Reeb graph of the same space.
We then construct an isomorphism between the mapper of $(\X,f)$ to the mapper of a super-level set of a probability density function concentrated on $(\X, f)$.
Finally, building on the approach of Bobrowski et al.~(2017), we show that, with high probability, we can recover the mapper of the super-level set given a sufficiently large sample.
Our work is the first to consider the mapper construction using the theory of cosheaves in a probabilistic setting. It is part of an ongoing effort to combine sheaf theory, probability, and statistics, to support topological data analysis with random data.

\end{abstract}

\section{Introduction}

In recent years, topological data analysis has been gaining momentum in  aiding knowledge discovery of large and complex data.
A great deal of work has been focused on data modeled as scalar fields.
For instance, scientific simulations and imaging tools produce data in the form of point cloud samples equipped with scalar values, such as temperature, pressure and grayscale intensity.
One way to understand and characterize the structure of a scalar field $f: \X \to \R$ is through various forms of topological descriptors, which provide meaningful and compact abstraction of the data.
Popular topological descriptors can be classified into vector-based ones such as persistence diagrams~\citep{EdelsbrunnerLetscherZomorodian2002} and barcodes~\citep{Ghrist2008, CarlssonZomorodianCollins2004}, graph-based ones such as Reeb graphs~\citep{Reeb1946} and their variants merge trees~\citep{BeketayevYeliussizovMorozov2014} and contour trees~\citep{CarrSnoeyinkAxen2003}, and complex-based ones such as Morse complexes, Morse-Smale complexes~\citep{GerberPotter2012,EdelsbrunnerHarerZomorodian2003,EdelsbrunnerHarerNatarajan2003}, and the mapper construction~\citep{SinghMemoliCarlsson2007}.

For a topological space $\X$ equipped with a function $f: \X \to \R$, the \emph{Reeb graph}, denoted as $\cR(\X,f)$, encodes the connected components of the level sets $f^{-1}(a)$ for $a$ ranging over $\R$.
It summarizes the structure of the data, represented as a pair $(\X, f)$, by capturing the evolution of the topology of its level sets.
Research surrounding Reeb graphs and their variants has been very active in recent years, from theoretical, computational and applications aspects, see~\cite{BiasottiGiorgiSpagnuolo2008} for a survey.
In the multivariate setting, Reeb spaces~\citep{EdelsbrunnerHarerPatel2008} generalize Reeb graphs and serve as topological descriptors of multivariate functions $f:\X \to \R^d$.
The Reeb graph is then a special case of a Reeb space for $d = 1$.

One issue with Reeb spaces are their limited applicability to point cloud data.
To facilitate their practical usage, a closely related construction called \emph{mapper}~\citep{SinghMemoliCarlsson2007} was introduced to capture the topological structure of a pair $(\X, f)$ (where $f:\X \to \R^d$).
Given a topological space $\X$ equipped with a $\R^d$-valued function $f$, for the classic mapper construction, we work with a finite good cover $\mathcal{U} = \{U_\alpha\}_{\alpha \in A}$ of $f(\X)$ for some indexing set $A$, such that $f(\X) \subseteq \bigcup{U_{\alpha}}$. 
Let $f^*(\mathcal{U})$ denote the cover of $\X$ obtained by considering the path-connected components of $f^{-1}(U_\alpha)$ for each $\alpha$. 
The mapper construction of $(\X, f)$ is defined to be the nerve of $f^*(\mathcal{U})$, denoted as $\mathcal{N}_{f^*(\mathcal{U})}$, see Figure~\ref{fig:enhanced-mapper}(h) for an example. 
By definition, the \emph{mapper} is an abstract simplicial complex; and its 1-dimensional skeleton is referred to as the classic \emph{mapper graph} in this paper.

As a computable alternative to the Reeb space, the mapper has enjoyed tremendous success in data science, including cancer research~\citep{NicolauLevineCarlsson2011} and sports analytics~\citep{Alagappan2012}; it is also a cornerstone of several  data analytics companies such as Ayasdi and Alpine Data Labs.
Many variants have been studied in recent years.
The \emph{$\alpha$-Reeb graph}~\citep{ChazalSun2014} redefines the equivalence relation between points using open intervals of length at most $\alpha$.
The \emph{multiscale mapper}~\citep{DeyMemoliWang2016} studies a sequence of mapper constructions by varying the granularity of the cover.
The \emph{multinerve mapper}~\citep{CarriereOudot2018} computes the multinerve~\citep{VerdiereGinotGoaoc2012} of the connected cover.
The \emph{Joint Contour Net} (JCN)~\citep{CarrDuke2013, CarrDuke2014} introduces quantizations to the cover elements by rounding the function values.
The \emph{extended Reeb graph}~\citep{BarralBiasotti2014} uses cover elements from a partition of the domain without overlaps.

Although the mapper construction has been widely appreciated by the practitioners, our understanding of its theoretical properties remains fragmentary.
Some questions important in theory and in practice center around its structure and its relation to the Reeb graph.
\begin{itemize}
\item [Q1.] \textbf{Information content:} What information is encoded by the mapper? How much information can we recover about the original data from the mapper by solving an inverse problem?
\item [Q2.] \textbf{Stability:} What is the structural stability of the mapper with respect to perturbations of its function, domain and cover?
\item [Q3.] \textbf{Convergence:} What is an appropriate metric under which the mapper converges to the Reeb graph as the number of sampled points goes to infinity and the granularity of the cover goes to zero?
\end{itemize}
To the best of our knowledge, our work is the first to address convergence in \textbf{a probabilistic setting}.
Given a mapper construction applied to points randomly sampled from a probability density function, we prove an asymptotic result: as the number of points $n \to \infty$, the mapper graph construction approximates that of the Reeb graph up to the granularity of the cover with high probability.

\para{Information, stability and convergence.}
We discuss our work in the context of related literature in topological data analysis.
As many topological descriptors, the mapper summarizes the information from the original data through a \emph{lossy} process.
To quantify its information content, Dey et al.~\citep{DeyMemoliWang2017} studied the topological information encoded by Reeb spaces, mappers and multi-scale mappers, where 1-dimensional homology of the mapper was shown to be no richer than the domain $\X$ itself.
Carri\'{e}re and Oudot~\citep{CarriereOudot2018} characterized the information encoded in the mapper using the extended persistence diagram of its corresponding Reeb graph.
Gasparovic et. al.~\citep{GasparovicGommelPurvine2018} provided full descriptions of persistent homology information of a metric graph via its intrinsic \v{C}ech complex, a special type of nerve complex.
In this paper, we study the information content of the mapper via a (co)sheaf-theoretic approach; in particular, through the notion of \emph{display locale}, we introduce an intermediate object called the \emph{enhanced mapper graph}, that is, a CW complex with weighted 0-cells.
We show that the enhanced mapper graph reduces the information loss during  summarization and may be of independent interest.

In terms of stability, Carri\'{e}re and Oudot~\citep{CarriereOudot2018} derived stability for the mapper graph using the stability of extended persistence diagrams equipped with the bottleneck distance under Hausdorff or Wasserstein perturbations of the data~\citep{CohenSteinerEdelsbrunnerHarer2009}.
Our work is similar to~\citep{CarriereOudot2018} in a sense that we study the stability of the enhanced mapper graph with respect to perturbation of the data $(\X, f)$, where the local stability depends on how the cover $\cU$ is positioned in relation to the critical values of $f$. However, we formalize the structural stability of the enhanced mapper graph using a categorification of the mapper algorithm and the interleaving distance of constructible cosheaves.

When $f$ is a scalar field and the connected cover of its domain $\R$ consists of a collection of open intervals, the mapper construction is conjectured to recover  the Reeb graph precisely as the granularity of the cover goes to zero~\citep{SinghMemoliCarlsson2007}.
Babu~\citep{Babu2013} studied the above convergence using levelset zigzag persistence modules and showed that the mapper converges to the Reeb graph in the bottleneck distance.
Munch and Wang~\citep{MunchWang2016} characterized the mapper using constructible cosheaves and proved the convergence between the (classic) mapper and the Reeb space (for $d \geq 1$) in interleaving distance. The enhanced mapper graph defined in this paper is similar to the geometric mapper graph introduced in \citep{MunchWang2016}. The differences between the enhanced mapper graph and geometric mapper consist of technical changes in the geometric realization of each space as a quotient of a disjoint union of closed intervals. Proposition \ref{prop:displaylocale} implies that the enhanced mapper graph is isomorphic to the display locale of the mapper cosheaf, giving theoretic significance to the geometrically realizable enhanced mapper graph. 

\cite{DeyMemoliWang2017} established a convergence result between the mapper and the domain under a Gromov-Hausdorff metric.
Carri\'{e}re and Oudot~\citep{CarriereOudot2018} showed convergence between the (multinerve) mapper and the Reeb graph using the functional distortion distance~\citep{BauerGeWang2014}. The enhanced mapper graph we define plays a role roughly analogous to the multinerve mapper in~\citep{CarriereOudot2018}, although with several important distinctions. Most significantly is the fact that the enhanced mapper graph is an $\R$-space, and as such is not a purely combinatorial object, in contrast to the multinerve mapper, which is a simplicial poset. 
Carri\'{e}re et al.~\citep{CarriereMichelOudot2018} proved convergence and provided a confidence set for the mapper using a bottleneck distance on certain extended persistence diagrams. They showed that the mapper is an optimal estimator of the Reeb graph and provided a statistical method for automatic parameter tuning using the rate of convergence.
 Like~\cite{CarriereMichelOudot2018}, this paper studies a notion of consistency (detailed below) for the mapper algorithm. In contrast to~\cite{CarriereMichelOudot2018}, the results provided here use the Reeb distance on constructible $\R$-graphs (defined in Section \ref{sec:Background}) rather than bottleneck distances on extended persistence diagrams, and are applicable to more general topological spaces (i.e.,~we do not require $\X$ to be a smooth manifold).

\para{Probabilistic mapper inference.}
This work is part of an effort to harness the theory of probability and statistics  to support and analyze the use of topological methods with random data.
To date, most of this effort has been put into problems related to the homology and persistent homology of random point clouds. The problem of \emph{homological inference} relates to the ability to recover the homology (or persistent homology) of an unknown space or function given random observations. In a noiseless setup this problem was studied in \cite{NiyogiSmaleWeinberger2008, Bobrowski2019,ChazalGlisseLabruere2015,KergorlayTillmannVipond2019,WangWang2018}.
The noisy setup was studied in \cite{NiyogiSmaleWeinberger2011,BobrowskiMukherjeeTaylor2017,ChazalFasyLecci2017,FasyLecciRinaldo2014}.
Briefly, these works provide methods to recover the homology, together with assumptions that guarantee  correct recovery with high probability.
In many of these, the results are asymptotic, taking the number of points $n\to \infty$.
The main reason for taking limits, is that the mathematics become more tractable, and provide simpler and more intuitive statements.
Such asymptotic results can be considered as proofs of \emph{consistency} for such homology estimation procedures. \myblue{In Section \ref{sec:Model}, we apply results of \cite{BobrowskiMukherjeeTaylor2017} to study consistency of the enhanced mapper construction introduced in Section \ref{sec:Background}. The statistical techniques we use are similar to those developed in \cite{ChazalGuibasOudotSkraba2011}. For further discussion of the differences between the techniques used in Section \ref{sec:Model} and the results of \cite{ChazalGuibasOudotSkraba2011}, see \cite{BobrowskiMukherjeeTaylor2017}.}

In a way, the work here uses similar ideas to perform ``mapper inference", a type of  \emph{structural inference}, and proves consistency.
Other probabilistic studies related to applied topology  mainly include     limiting theorems (laws of large numbers, and central limit theorems), and extreme value analysis for the homology and persistent homology of random data (see e.g. \cite{YogeshwaranSubagAdler2016,HiraokaShiraiTrinh2018,OwadaAdler2017,BobrowskiKahleSkraba2017,KahleMeckes2013}). However, these are much more detailed quantitative statements than what we are looking for when working with the mapper construction.

\para{Contributions.}
We highlight four contributions of this paper. 
\begin{itemize}
\item First, in Section \ref{sec:reeb}, we introduce and construct an \emph{enhanced mapper graph}. This graph retains more geometric information about the underlying space than the combinatorially defined classic mapper graph, multinerve mapper graph, and geometric mapper graph (defined in~\cite{MunchWang2016}). Moreover, we show that the enhanced mapper graph construction provides a concrete realization of the display locale of a constructible cosheaf. 
\item Second, in Section \ref{sec:CategorifiedMapper}, we give a categorical interpretation of the mapper construction. This categorification allows us to view mapper construction as a functor from the category of cosheaves to the category of constructible cosheaves. We can recover a geometric realization of the mapper construction from the categorical realization by taking enhanced mapper graphs, i.e., the display locales, of the corresponding constructible cosheaves. 
\item Third, we prove convergence (Theorem \ref{thm:interleave}) and stability (Theorem \ref{thm:stability}) for the mapper cosheaf in the interleaving distance. 
\item Finally, we obtain results on the approximation quality of random mapper graphs obtained from noisy data on spaces which are not assumed to be manifolds (Theorem \ref{thm:main}). 
\end{itemize}
Moreover, using the results of~\cite{deSilvaMunchPatel2016}, each of our theorems are reinterpreted in terms of the geometrically-defined enhanced mapper graph and Reeb distance on $\R$-graphs. This reinterpretation allows us to state our main result below without referring to the machinery of cosheaf theory.

\noindent\textbf{Theorem} (Corollary \ref{cor:main}) \textit{
 Let $\cR(\X,f)$ be the Reeb graph of a constructible $\R$-space $(\X,f)$, $\hat{\mathfrak{D}}^\pi_n$ be the enhanced mapper graph associated to the cosheaf $\hat{\sD}^\pi_n$ defined in Section \ref{sec:MainResults}, and $d_R(\cdot,\cdot)$ be the Reeb distance defined in Section \ref{sec:Background}. Using the notation defined in Section \ref{sec:Model}, if there exists $\epsilon<\delta_\cU$ such that $p$ is $\epsilon$-concentrated on $\X$, then
  \[  \limninf\prob{d_R\big(\hat{\mathfrak{D}}^\pi_n,\cR(\X,f)\big)\le\text{res}_f\cU} =1.\]}
Intuitively speaking, the above theorem states that we can recover (a variant of) the mapper graph using the theory of cosheaves in a probabilistic setting. 
In particular,  with high probability, the distance between an enhanced mapper graph and the Reeb graph is upper bounded by the resolution of the cover (denoted as $\text{res}_f\cU$, see Definition \ref{def:resolution}) as the number of samples goes to infinity. The proof of the theorem relies on two preliminary results. First, in Theorem \ref{thm:interleave}, we construct an interleaving between the Reeb cosheaf and mapper cosheaf. Proposition \ref{prop:level2-box} is the second key ingredient of the proof, giving a probabilistic recovery of the mapper cosheaf from random points. By interpreting the enhanced mapper graph in terms of cosheaf theory, we are able to simplify many of the proofs for convergence and stability. Generally, this paper illustrates  the utility of combining sheaf theory with statistics in order to study  robust topological and geometric properties of data. 

\para{Pictorial overview.} 
To better illustrate our key constructions, we give an example of an enhanced mapper graph. 
As illustrated in Figure~\ref{fig:enhanced-mapper}, given a topological space equipped with a height function $(\X, f)$, we are interested in studying how well its classic mapper graph (h) (with a fixed cover) approximates its  Reeb graph (b). 
In order to study this problem, we construct a categorification of the mapper graph, through the theory of constructible cosheaves (d). The display locale functor is used to recover a geometric object from these category-theoretic constructible cosheaves. The geometric realization of the display locale of the mapper cosheaf is referred to as the enhanced mapper graph (g). We outline an explicit geometric realization of the enhanced mapper graph as a quotient of a disjoint union of closed intervals (f).  

The main result of the paper, Theorem \ref{thm:main}, gives (with high probability) a bound on the interleaving distance between the Reeb cosheaf and the enhanced mapper cosheaf.
In order to interpret this result in terms of probabilistic convergence (Corollary \ref{cor:main}), we apply the display locale functor to obtain the Reeb graph and the enhanced mapper graph from their cosheaf-theoretic analogues. This procedure results  (with high probability) in a bound on the Reeb distance between an enhanced mapper graph and the Reeb graph of a constructible $\R$-space with random data.  

\begin{figure}[!ht]
\begin{center}
\includegraphics[width=0.85\textwidth]{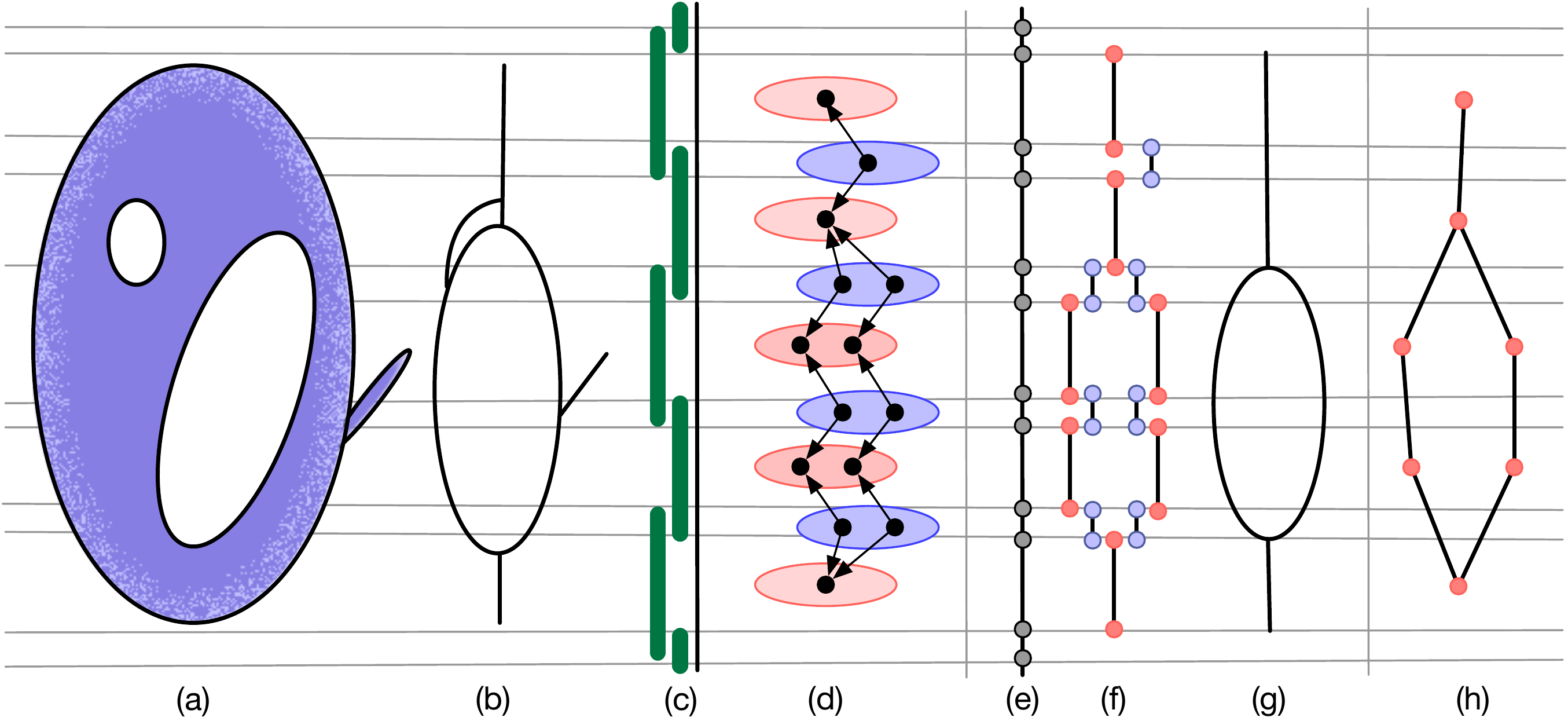} 
\caption{An example of an enhanced mapper graph. 
(a) An $\R$-space $(\X, f)$ given by a topological space $\X$ (in blue) equipped with a height function $f: \X \to \R$. 
(b) Reeb graph of $(\X, f)$. (c) Nice cover of $\R$ with open intervals. (d) Visualization of the mapper cosheaf. (e) Stratification of $\R$. (f) Disjoint union of closed intervals ($\widetilde{\mathfrak{D}}$, in the notation of Section \ref{sec:reeb}), with quotient isomorphic to the enhanced mapper graph. (g) Enhanced mapper graph ($\mathfrak{D}$, in the notation of Section \ref{sec:reeb}). (h) Classic mapper graph of $(\X, f)$.}
\label{fig:enhanced-mapper}
\end{center}
\end{figure}


\section{Background}
\label{sec:Background}

In this section, we review the results of~\cite{deSilvaMunchPatel2016} together with~\cite{MunchWang2016}, showing that the interleaving distance between the mapper of the constructible $\R$-space $(\X,f)$ relative to the open cover $\mathcal{U}$ of $\R$ and the Reeb graph of $(\X,f)$ is bounded by the resolution of the open cover. Motivated by the categorification of Reeb graphs in~\cite{deSilvaMunchPatel2016}, we introduce a categorified mapper algorithm, and restate the main results of~\cite{MunchWang2016} in this framework. 

Categorification, in this context, means that we are interested in using the theory of constructible cosheaves to study Reeb graphs and mapper graphs.
We can accomplish this by defining a cosheaf (the Reeb cosheaf) whose display locale is isomorphic to a given Reeb graph.
One goal (completed in \citep{deSilvaMunchPatel2016}) of this approach is to use cosheaf theory to define an extended metric on the category of Reeb graphs.
A natural candidate from the perspective of cosheaf theory is the interleaving distance.
Suppose we want to use the interleaving distance of cosheaves to determine if two Reeb graphs are homeomorphic.
We can first think of each Reeb graph as the display locale of a cosheaf, $\sF$ and $\sG$, respectively.
This allows us to rephrase our problem as that of determining if the cosheaves, $\sF$ and $ \sG$, are isomorphic.
In general, interleaving distances cannot answer this question, since the interleaving distance is an extended \emph{pseudo}-metric on the category of all cosheaves.
In other words, having interleaving distance equal to 0 is not enough to guarantee that $\sF$ and $\sG$ are isomorphic as cosheaves.
This seems to suggest that the interleaving distance is insufficient for the study of Reeb graphs.
However (due to results of \citep{deSilvaMunchPatel2016}), if we restrict our study to the category of constructible cosheaves (over $\R$), we can avoid this subtlety. The interleaving distance is in fact an extended \emph{metric} on the category of constructible cosheaves. If two constructible cosheaves have interleaving distance equal to 0, then they are isomorphic as cosheaves. Therefore, the display locales of constructible cosheaves (over $\R$) are homeomorphic if the interleaving distance between the cosheaves is equal to 0. In other words, if we want to know if two Reeb graphs are homeomorphic, it is sufficient to consider the interleaving distance between constructible cosheaves $\sF$ and $ \sG$, provided that the display locales of the constructible cosheaves recover the Reeb graphs. Therefore, in the remainder of this section, we  define a mapper cosheaf, and show that the Reeb cosheaf of a constructible $\R$-space is a constructible cosheaf, and that the mapper cosheaves are constructible. This allows us to use the commutativity of diagrams and the interleaving distance to prove convergence of the corresponding display locales, that is, the Reeb graphs and the enhanced mapper graphs. 
We use the example in Figure~\ref{fig:enhanced-mapper} as a reference for various notions.

\subsection{Constructible $\R$-spaces}
\label{sec:R_Spaces}
We begin by defining constructible $\R$-spaces, which we consider to be the underlying spaces for estimating the Reeb graphs, see Figure~\ref{fig:enhanced-mapper}. 
Constructible $\R$-spaces can be considered as a class of topological spaces which provide a natural setting for generalizing aspects of classical Morse theory to the study of singular spaces. Like smooth manifolds equipped with a Morse function, constructible $\R$-spaces are topological spaces equipped with a real valued function $f$, whose fibers, $f^{-1}(x)$, satisfy certain regularity conditions. Specifically, the topological structure of the fibers of the real valued function are required to only change at a finite set of function values. The function values which mark changes in the topological structure of fibers are referred to as critical values.  

\begin{definition}[\cite{deSilvaMunchPatel2016}]
An \emph{$\R$-space} is a pair $(\X,f)$, where $\X$ is a topological space and $f:\X\rightarrow \R$ is a continuous map.
A \emph{constructible $\R$-space} is an $\R$-space $(\X,f)$ satisfying the following conditions:
\begin{enumerate}\denselist
    \item There exists a finite increasing sequence of points $S=\{a_0,\cdots,a_n\}\subset\R $, two finite sets of locally path-connected spaces $\{\V_0,\cdots,\V_n\}$ and $\{\E_0,\cdots, \E_{n-1}\}$, and two sets of continuous maps $\{\ell_i:\E_i\rightarrow \V_i\}$ and $\{r_i:\E_i\rightarrow \V_{i+1}\}$, such that $\X$ is the quotient space of the disjoint union
    \[
    \coprod_{i=0}^n \V_i\times \{a_i\}\sqcup \coprod_{i=0}^{n-1}\E_i\times [a_i,a_{i+1}]
    \]
    by the relations
    \[
    (\ell_i(x),a_i)\sim (x,a_i)\text{  and  } (r_i(x),a_{i+1})\sim (x,a_{i+1})
    \]
    for all $i$ and $x\in \E_i$.

    \item The continuous function $f:\X\rightarrow \R$ is given by projection onto the second factor of $\X$.
\end{enumerate}

These are the objects of categories $\Rspace$ and $\Rspacec$, consisting of $\R$-spaces and constructible $\R$-spaces, respectively.
Morphisms in these categories are function-preserving maps; that is,~$\phi:(\X,f) \to (\Y,g)$ is given by a continuous map $\phi:\X \to \Y$ such that $g \circ \phi(x) = f(x)$.
\end{definition}

\begin{example}
A smooth compact manifold $\X$ with a Morse function $f$ constitutes a constructible $\R$-space. 
For instance, Figure~\ref{fig:enhanced-mapper}(a) illustrates a topological space $\X$ equipped with a height function $f$; the pair $(\X, f)$ is an $\R$-space. 
Similarly, a height function $f$ on a torus $\X$ gives rise to an $\R$-space $(\X, f)$ in Figure~\ref{fig:torus-enhanced-mapper}(a).
\end{example}
In fact, $\X$ is not required to be a manifold for $(\X, f)$ to be an $\R$-space. 
Throughout the remainder of this paper, we assume that $(\X, f)$ is a constructible $\R$-space.
\begin{definition}[\cite{deSilvaMunchPatel2016}]
An $\emph{$\R$-graph}$ is a constructible $\R$-space such that the sets $\V_i$ and $\E_i$ are finite sets (with the discrete topology) for all $i$.
\end{definition}
\begin{example}
The Reeb graph of a constructible $\R$-space is an $\R$-graph. For instance, the Reeb graph of $(\X, f)$ in Figure~\ref{fig:enhanced-mapper}(b) is an $\R$-graph. 
Similarly, the Reeb graph of a Morse function on a torus is an $\R$-graph, see Figure~\ref{fig:torus-enhanced-mapper}(b).
\end{example}

\subsection{Constructible cosheaves}
\label{ssec:ConstructibleCosheaves}

Sheaves and cosheaves are category-theoretic structures, called functors, which provide a framework for associating data to open sets in a topological space. These associations are required to preserve certain properties inherent to the topology of the space. In this way, one can study the topological structure of the space by studying the data associated to each open set by a given sheaf or cosheaf. In the following sections, we will use cosheaves to encode information about a constructible $\R$-space by associating open intervals in the real line to sets of (path-)connected components of fibers of the real valued function corresponding to the constructible $\R$-space. 

Let $\Int$ be the category of connected open sets in $\R$ with inclusions which we refer to as intervals, and $\Set$ the category of abelian groups with group homomorphism maps. We first define a cosheaf over $\R$, which we propose to be the natural objects for categorifying the mapper algorithm. 

\begin{definition}
A \emph{pre-cosheaf} $\sF$ on $\R$ is a covariant functor $\sF: \Int \to \Set$.
The category of precosheaves on $\R$ is denoted $\Set^\Int$ with morphisms given by natural transformations.

A pre-cosheaf $\sF$ is a \emph{cosheaf} if
\[
\varinjlim_{V\in\cV}\sF(V) = \sF(U)
\]
for each open interval $U \in \Int$ and each open interval cover $\cV\subset \Int$ of $U$, which is closed under finite intersections.
The full subcategory of $\Set^\Int$ consisting of cosheaves is denoted $\Csh$.
\end{definition}

\begin{remark}
\myblue{
We note that usually, cosheaves are defined over the category of arbitrary open sets rather than the category of connected open sets. 
However, the category of cosheaves defined over connected open sets is equivalent to the category of cosheaves defined over arbitrary open sets, by the colimit property of cosheaves. 
When we define smoothing operations on cosheaves in Section \ref{ssec-interleavings}, there are important distinctions that will make clear the need for the definition with respect to $\Int$, as set-thickening operations do not preserve the cosheaf property otherwise.
}
\end{remark}

Since we are interested in working with cosheaves which can be described with a finite amount of data, we will restrict our attention to a well-behaved subcategory of $\Csh$, consisting of constructible cosheaves (defined below). Constructibility can be thought of as a type of ``tameness" assumption for sheaves and cosheaves. 

\begin{definition}
A cosheaf $\sF$ is \emph{constructible} if there exists a finite set $S\subset \R$ of \emph{critical values} such that $\sF[U\subset V]$ is an isomorphism whenever $S\cap U = S\cap V$.
The full subcategory of $\Csh$ consisting of constructible cosheaves is denoted $\Cshc$.
\end{definition}

\subsection{The Reeb cosheaf and display locale functors}
\label{sec:reeb}
We introduce the Reeb cosheaf and display locale functors. These functors relate the category of constructible cosheaves to the category of $\R$-graphs, and provide a natural categorification of the Reeb graph~\citep{deSilvaMunchPatel2016}. In other words, via both Reeb cosheaf functor and display locale functors, one could consider the translation between the data and their corresponding categorical interpretations.

Let $\sR_f$ be the \emph{Reeb cosheaf} of $(\X,f)$ on $\R$, defined by
$$
\sR_f(U)=\pi_0(\X^U),
$$
where $\X^U :=  f^{-1}(U)$ and $\pi_0(\X^U)$ denotes the set of path components of $\X^U$.  
\begin{definition}
The \emph{Reeb cosheaf functor} $\cC$ from the category of constructible $\R$-spaces to the category of constructible cosheaves
\begin{equation*}
  \begin{tikzcd}
\Rspacec \ar[r, "\cC"]& \Cshc
  \end{tikzcd}
\end{equation*}
is defined by $\cC((\X,f))=\sR_f$.
For a function-preserving map $\phi:(\X,f) \to (\Y,g)$, the resulting morphism $\cC[\phi]$
is given by $\cC[\phi]: \sR_f(U) = \pi_0 \circ f\inv(U) \to \pi_0 \circ g\inv(U)= \sR_g(U)$ induced by $\phi \circ f\inv(U) \subseteq g\inv(U)$.

\end{definition}

\begin{definition}\label{defn:costalk}
The \emph{costalk} of a (pre-)cosheaf $\sF$ at $x\in \R$ is\myblue{
$$
\sF_x=\varprojlim_{I\ni x}\sF(I).
$$
For each costalk $\sF_x$, there is a natural map $\sF_x\rightarrow \sF(I)$ (given by the universal property of limits)} for each open interval $I$ containing $x$.
\end{definition}
In order to related the Reeb and mapper cosheaves to geometric objects, we  make use of the notion of \emph{display locale}, introduced in~\citep{Funk1995}.
\begin{definition}
The \emph{display locale} of a cosheaf $\sF$ (as a set) is defined as
$$
\mathcal{D}(\sF)=\coprod_{x\in \R}\sF_x.
$$
A topology on $\cD(\sF)$ is generated by open sets of the form
$$
U_{I,a}=\{s\in\sF_x:x\in I \text{ and }s\mapsto a\in \sF(I)\},
$$
for each open interval $I\in \Int$ and each section $a\in\sF(I)$.
\end{definition}
The display locale gives a functor from the category of cosheaves to the category of $\R$-graphs,
\begin{equation*}
  \begin{tikzcd}
  \Cshc \ar[r,"\cD"]  & \Rgraph.
  \end{tikzcd}
\end{equation*}

We proceed by giving an explicit geometric realization of the display locale of a constructible cosheaf. Let $\sF$ be a constructible cosheaf with set of critical values $\R_0 \subset \R$. Let $\R_1 = \R \setminus \R_0$ be the complement of $\R_0$, so that we form a stratification
\[\R=\R_0\sqcup \R_1,\] 
See Figure~\ref{fig:enhanced-mapper}(e) for an example (black points are in $\R_0$, their complements are in $\R_1$). 
Let $S_1$ be the set of connected components of $\R_1$, i.e., the 1-dimensional stratum pieces. For $x\in\R_0$, let $I_x$ denote the largest open interval containing $x$ such that $I_{x}\cap \R_0= \{x\}$. Let
\[\tilde{\mathfrak{D}}(\sF):=\coprod_{V\in S_1}\overline{V}\times \sF(V)\sqcup\coprod_{x\in \R_0}\{x\}\times \sF(I_x) ,\]
where $\overline{V}$ is the closure of $V$ and the product $C\times\emptyset$ of a set $C$ with the empty set is understood to be empty. Geometrically, $\tilde{\mathfrak{D}}(\sF)$ is a disjoint union of connected closed subsets of $\R$; if the support of $\sF$ is compact, then $\tilde{\mathfrak{D}}(\sF)$ is a disjoint union of closed intervals and points. Let $\pi$ denote the projection map
\begin{eqnarray*}
\pi:\tilde{\mathfrak{D}}(\sF)&\rightarrow& \R\\
(x,a)&\mapsto & x.
\end{eqnarray*}
Suppose $(x,a)\in\overline{V}\times \sF(V)\subset \tilde{\mathfrak{D}}(\sF)$ and $x\in\R_0$. We have that $V\cap \R_0=\emptyset$ and $I_{x}\cap V\neq \emptyset$ (because $x$ lies on the boundary of $V$). By maximality of $I_{x}$, we have the inclusion $V\subset I_{x}$. Let $\varphi_{(x,a)}$ be the map
\begin{eqnarray*}
\varphi_{(x,a)}:\sF(V)&\rightarrow& \sF(I_{x})
\end{eqnarray*}
induced by the inclusion $V\subset I_{x}$. We can extend this map to the fiber of $\pi$ over $x$,
\begin{eqnarray*}
\psi_{x}:\pi^{-1}(x)&\rightarrow& \sF(I_{x}),
\end{eqnarray*}
where $\psi_x((x,a)):=\phi_{(x,a)}(a)$ if $(x,a)\in \overline{V}\times \sF(V)$ and $\psi_x((x,a)):= a$ if $(x,a)\in \{x\}\times \sF(I_x)$. Finally, we define an equivalence relation of points in $\tilde{\mathfrak{D}}(\sF)$. Suppose $(x,a),(y,b)\in\tilde{\mathfrak{D}}(\sF)$. Then $(x,a)\sim(y,b)$ if
\begin{enumerate}\denselist
    \item $x=y\in\R_0$, and
    \item $\psi_x(a)=\psi_x(b)\in \sF(I_x)$.
\end{enumerate}
Finally, let
\[
\mathfrak{D}(\sF): = \tilde{\mathfrak{D}}(\sF)/\sim
\]
be the quotient of $\tilde{\mathfrak{D}}(\sF)$ by the equivalence relation. The projection $\pi$ factors through the quotient, giving a map $\bar{\pi}:\mathfrak{D}(\sF)\rightarrow\R$.
\begin{proposition}\label{prop:displaylocale}
If $\sF$ is a constructible cosheaf with set of critical values $S$, then $\mathfrak{D}(\sF)$ is a 1-dimensional CW-complex which is isomorphic (as an $\R$-space) to the display locale, $\cD(\sF)$, of $\sF$.
\end{proposition}
\begin{proof}
We will construct a homeomorphism $\gamma:\mathfrak{D}(\sF)\rightarrow \cD(\sF)$ which preserves the natural quotient maps $\bar{f}:\cD(\sF)\rightarrow \R$ and $\bar{\pi}:\mathfrak{D}(\sF)\rightarrow \R$. Given $x\in\R_1$, we have that $\bar{\pi}^{-1}(x)= \{x\}\times\sF(V)$, where $V$ is the connected component of $\R_1$ which contains $x$. Since $\sF$ is constructible with respect to the chosen stratification, we have that $\sF(V)\cong\sF_x$. This gives a bijection from $\bar{\pi}^{-1}(x)$ to $\bar{f}^{-1}(x)$. For $x\in\R_0$, the fiber $\bar{\pi}^{-1}(x)$ is by construction in bijection with $ \sF(I_x)$. Again, since $\sF$ is constructible and $I_x\cap \R_0 = B(x)\cap \R_0$ for each sufficiently small neighborhood $B(x)$ of $x$, we have that $\sF(I_x)\cong \sF_x$. These bijections define a map $\gamma:\mathfrak{D}(\sF)\rightarrow \cD(\sF)$, which preserves the quotient maps by construction. All that remains is to show that $\gamma$ is continuous. 

Suppose $x\in\R_1$, and let $V$ be the connected component of $\R_1$ which contains $x$, and $B(x)$ be an open neighborhood of $x$ such that $B(x)\subset V$. Then $\sF_y\cong \sF(V)$ for each $y\in B(x)$, and $\sF(B(x))\cong \sF(V)$. Recall the definition of the basic open sets $U_{I,a}$ in the definition of display locale (with notation adjusted to better align with the current proof),
$$
U_{I,a}=\left\{s\in\sF_y\subset \coprod_{x\in \R}\sF_x:y\in I \text{ and }s\mapsto a\in \sF(I)\right\}.
$$
Using the above isomorphisms to simplify the definition according to the current set-up, we get 
\begin{eqnarray*}
U_{B(x),a}&\cong&\left\{a\in\coprod_{y\in B(x)} \sF(V)\right\}.
\end{eqnarray*}
Therefore, $\gamma^{-1}(U_{B(x),a})=B(x)\times \{a\}$, which is open in the quotient topology on $\mathfrak{D}(\sF)$. 

Suppose $x\in \R_0$, and let $B(x)$ be a neighborhood of $x$ such that $B(x)\subset I_x$. Let $V_1$ and $V_2$ denote the two connected components of $\R_1$ which are contained in $I_x$. If $y\in B(x)$, then $\sF_y$ is isomorphic to either $\sF(V_1)$, $\sF(V_2)$, or $\sF(I_x)$. Moreover, since $\sF$ is constructible, we have that $\sF(B(x))\cong \sF(I_x)$. Let $a'\in \sF(I_x)$ correspond to $a\in\sF(B(x))$ under the isomorphism $\sF(I_x)\cong \sF(B(x))$. Following the definitions, we have that 
\begin{align*}
& \pi^{-1}\left(\gamma^{-1}(U_{B(x),a})\right)  = \\
& \left(\overline{V_1}\cap B(x)\right)\times \sF[V_1\subset I_x]^{-1}(a') \sqcup \left( \overline{V_2}\cap B(x)\right)\times \sF[V_2\subset I_x]^{-1}(a') \sqcup \{x\}\times \{a'\},
\end{align*}
where $\sF[V_i\subset I_x]^{-1}(a')$ is understood to be a (possibly empty) subset of $\sF(V_i)$. It follows that $\gamma^{-1}(U_{B(x),a})$ is open in the quotient topology on $\mathfrak{D}(\sF)$. Therefore, $\gamma^{-1}$ maps open sets to open sets, and we have shown that $\gamma$ is a homeomorphism which preserves the quotient maps $\bar{f}$ and $\bar{\pi}$, i.e., $\bar{f}(\gamma((x,a)))=\bar{\pi}((x,a))=x$. 
\qed
\end{proof}
It follows from the proposition that $\mathfrak{D}(\sF)$ is independent (up to isomorphism) of choice of critical values $\R_0$. Additionally, we now note that we can freely use the notation $\mathfrak{D}(\sF)$ or $\cD(\sF)$ to refer to the display locale of a constructible cosheaf over $\R$. 
We will continue to use both symbols, reserving $\cD$ for the display locale of an arbitrary cosheaf, and using $\mathfrak{D}$ when we want to emphasize the above equivalence for constructible cosheaves.

In \citep{deSilvaMunchPatel2016}, it is shown that the Reeb graph $\cR(\X,f)$ of $(\X,f)$ is naturally isomorphic to the display locale of $\sR_f$. Moreover, the display locale functor $\mathcal{D}$ and the Reeb functor $\cC$ are inverse functors and define an equivalence of categories between the category of Reeb graphs and the category of constructible cosheaves on $\R$. This equivalence is closely connected to the more general relationships between constructible cosheaves and stratified coverings studied in~\citep{Woolf2009}. The result allows us to define a distance between Reeb graphs by taking the interleaving distance between the associated constructible cosheaves as shown in the following section.

\subsection{Interleavings}
\label{ssec-interleavings}

We start by defining the interleavings on the categorical objects. Interleaving is a typical tool in topological data analysis for quantifying proximity between objects such as persistence modules and cosheaves.
For $U \subseteq \R$, let $U \mapsto U_\e := \{ y \in \R \mid \|y-U\| \leq \e\}$.
If $U = (a,b) \in \Int$, then $U_\e = (a-\e, b+\e)$.
\begin{definition}
Let $\sF$ and $\sG$ be two cosheaves on $\R$.
An \emph{$\varepsilon$-interleaving} between $\sF$ and $\sG$ is given by two families of maps
\[
\varphi_U:\sF(U)\rightarrow\sG(U_\varepsilon),\quad \psi_U:\sG(U)\rightarrow \sF(U_\varepsilon)
\]
which are natural with respect to the inclusion $U\subset U_\varepsilon$, and such that
\[
\psi_{U_\varepsilon}\circ\varphi_U = \sF[U\subset U_{2\varepsilon}],\quad \varphi_{U_\varepsilon}\circ\psi_U=\sG[U\subset U_{2\varepsilon}]
\]
for all open intervals $U\subset \R$.
Equivalently, we require that the diagram
\begin{equation*}
  \begin{tikzcd}
  \sF(U)
    \ar[r]
    \ar[dr, "\phi_U", very near start, outer sep = -2pt]
  & \sF(U_\e)
    \ar[r]
    \ar[dr, "\phi_{U_\e}", very near start, outer sep = -2pt]
  & \sF(U_{2\e})
  \\
  \sG(U)
    \ar[r]
    \ar[ur, crossing over, "\psi_U"', very near start, outer sep = -2pt]
  & \sG(U_\e)
    \ar[r]
    \ar[ur, crossing over, "\psi_{U_\e}"', very near start, outer sep = -2pt]
  & \sG(U_{2\e})
  \end{tikzcd}
\end{equation*}
commutes, where the horizontal arrows are induced by $U \subseteq U_\e \subseteq U_{2\e}$.

The \emph{interleaving distance} between two cosheaves $\sF$ and $\sG$ is given by
\[
d_I(\sF,\sG):=\inf\{\varepsilon\mid \text{ there exists an $\e$-interleaving between $\sF$ and $\sG$}\}.
\]
\end{definition}


Now that we have an interleaving for elements of $\Cshc$ along with an equivalence of categories between $\Cshc$ and $\Rgraph$, we can develop this into an interleaving distance for the Reeb graphs themselves. The interleaving distance for Reeb graphs will be defined using a smoothing functor, which we construct below. 
\begin{definition}
Let $(\X,f)$ be a constructible $\R$-space.
For $\varepsilon\ge 0 $, define the \emph{thickening functor} $\cT_\varepsilon$ to be
\[
\cT_\varepsilon(\X,f)=(\X\times [-\varepsilon,\varepsilon],f_\varepsilon),
\]
where $f_\varepsilon(x,t)=f(x)+t$. 
Given a morphism $\alpha:\X\rightarrow \Y$,
\begin{eqnarray*}
\cT_\varepsilon(\alpha):\X\times [-\varepsilon,\varepsilon]&\rightarrow & \Y\times [-\varepsilon,\varepsilon]\\
(x,t)&\mapsto& (\alpha(x),t).
\end{eqnarray*}
The \emph{zero section map} is the morphism $(\X,f)\rightarrow \cT_\varepsilon(\X,f)$ induced by
\begin{eqnarray*}
\X&\rightarrow & \X\times [-\varepsilon,\varepsilon]\\
x&\mapsto & (x,0).
\end{eqnarray*}
\end{definition}

\begin{proposition}[{\cite[Proposition~4.23]{deSilvaMunchPatel2016}}]
 The thickening functor $\cT_\e$ maps $\R$-graphs to constructible $\R$-spaces, i.e., if $(\G,g)\in \R \text{-\bf{{graphs}}}$ then $\cT_\varepsilon(\G,g)\in\R \text{-\bf{{spaces}}}^{\text{\bf{{c}}}}$.
\end{proposition}

In general, the thickening functor $\cT_\varepsilon$ will output a constructible $\R$-space, and not an $\R$-graph. In order to define a `smoothing' functor for $\R$-graphs (following \cite{deSilvaMunchPatel2016}), we need to introduce a Reeb functor, which maps a constructible $\R$-space to an $\R$-graph. 
\begin{definition}
The \emph{Reeb graph functor} $\cR$ maps a constructible $\R$-space $(\X,f)$ to an $\R$-graph $(\X_f,\bar{f})$, where $\X_f$ is the Reeb graph of $(\X,f)$ and $\bar{f}$ is the function induced by $f$ on the quotient space $\X_f$. The \emph{Reeb quotient map} is the morphism $(\X,f)\rightarrow \cR(\X,f)$ induced by the quotient map $\X\rightarrow \X_f$.
\end{definition}
Now we can define a smoothing functor on the category of $\R$-graphs. 

\begin{definition}
Let $(\G,f) \in \Rgraph$.
The \emph{Reeb smoothing functor} $\cS_\e:\Rgraph \rightarrow \Rgraph$ is defined to be the Reeb graph of an $\e$-thickened $\R$-graph
\[
\cS_\e(\G,f)= \cR\left( \cT_\e(\G,f)\right).
\]
\end{definition}

The Reeb smoothing functor $\cS_\e$ defined above is used to define an interleaving distance for Reeb graphs, called the Reeb interleaving distance. The Reeb interleaving distance, defined below, can be thought of as a geometric analogue of the interleaving distance of constructible cosheaves. Let $\zeta_\F^\e$ be the map from $(\F,f)$ to $\cS_\e(\F,f)$ given by the composition of the zero section map $(\F,f)\rightarrow \cT_\e(\F,f)$ with the Reeb quotient map $\cT_\e(\F,f)\rightarrow \cR(\cT_\e(\F,f))$. To ease notation, we will denote the composition of $\zeta_\F^\e:(\F,f)\rightarrow \cS_\e(\F,f)$ with $\zeta_{\cS_\e(\F,f)}:\cS_\e(\F,f)\rightarrow \cS_\e(\cS_\e(\F,f))$ by $\zeta_\F^\e(\zeta_\F^\e(\F,f))$. 

\begin{definition}
Let $(\F,f)$ and $(\G,g)$ be $\R$-graphs.
We say that $(\F,f)$ and $(\G,g)$ are \emph{$\varepsilon$-interleaved} if there exists a pair of function-preserving  maps
\[
\alpha:(\F,f)\rightarrow \cS_\varepsilon(\G,g)\qquad\text{and}\qquad \beta:(\G,g)\rightarrow \cS_\varepsilon(\F,f)
\]
such that
\[
\cS_\e(\beta)\left(\alpha(\F,f)\right) = \zeta_\F^\e\left(\zeta_\F^\e(\F,f)\right)\quad\text{and}\quad \cS_\e(\alpha)\left(\beta(\G,g)\right) = \zeta_\G^\e\left(\zeta_\G^\e(\G,g)\right).
\]
That is, the diagram
\begin{equation*}
  \begin{tikzcd}
  (\F,f)
    \ar[r]
    \ar[dr, "\alpha", very near start, outer sep = -2pt]
  & \zeta_\F^\e(\F,f)
    \ar[r]
    \ar[dr, "\cS_\e(\alpha)", very near start, outer sep = -2pt]
  & \zeta_\F^\e\left(\zeta_\F^\e(\F,f)\right)
  \\
  (\G,g)
    \ar[r]
    \ar[ur, crossing over, "\beta"', very near start, outer sep = -2pt]
  & \zeta_\G^\e(\G,g)
    \ar[r]
    \ar[ur, crossing over, "\cS_\e(\beta)"', very near start, outer sep = -2pt]
  & \zeta_\G^\e\left(\zeta_\G^\e(\G,g)\right)
  \end{tikzcd}
\end{equation*}
commutes.

The \emph{Reeb interleaving distance}, $d_R\left((\F,f),(\G,g)\right)$, is defined to be the infimum over all $\varepsilon$ such that there exists an $\varepsilon$-interleaving of $(\F,f)$ and $(\G,g)$:
\[
d_R\left((\F,f),(\G,g)\right):=\inf\{\varepsilon:\text{ there exists an $\varepsilon$-interleaving of $(\F,f)$ and $(\G,g)$}\}.
\]
\end{definition}

\begin{remark} 
We should remark on a technical aspect of the above definition. The composition $\zeta_\F^\e\circ\zeta_\F^\e(\F,f)$ is naturally isomorphic to $\zeta_\F^{2\e}(\F,f)$. However, since the definition of the Reeb interleaving distance requires certain diagrams to commute, it is necessary to specify \emph{an} isomorphism between $\zeta_\F^\e\circ\zeta_F^\e(\F,f)$ and $\zeta_\F^{2\e}(\F,f)$ if one would like to replace $\zeta_\F^\e\circ\zeta_F^\e(\F,f)$ with $\zeta_\F^{2\e}(\F,f)$ in the commutative diagrams. Therefore, we choose to work exclusively with the composition of zero section maps, rather than working with diagrams which commute up to natural isomorphism. 
\end{remark}

The remaining proposition of this section gives a geometric realization of the interleaving distance of constructible cosheaves.
\begin{proposition}[\cite{deSilvaMunchPatel2016}]
 $\cD(\sF)$ and $\cD(\sG)$ are $\e$-interleaved as $\R$-graphs if and only if $\sF$ and $\sG$ are $\e$-interleaved as constructible cosheaves.
\end{proposition}

\begin{figure}
    \centering
    \includegraphics[width = .5\textwidth]{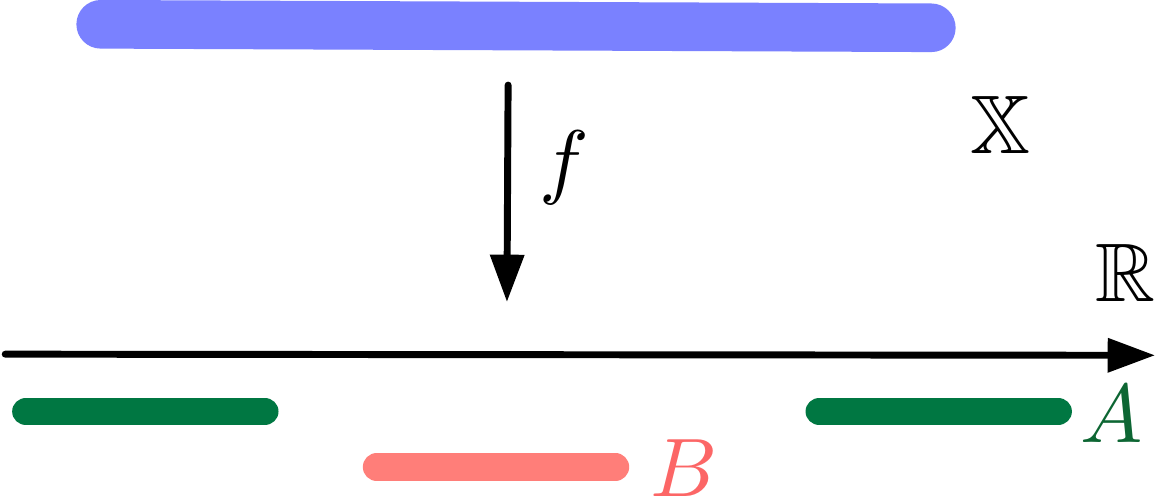}
    \caption{A counter example showing why we use $\mathbf{Int}$ rathar than  $\mathbf{Open}(\mathbb{R})$ for the definition of cosheaves in our context. See  Remark \ref{rem:counterexample}.}
    \label{fig:counterExample}
\end{figure}

\begin{remark}
\label{rem:counterexample}
\myblue{
Cosheaves are usually defined as functors on the category of open sets instead of functors on the connected open sets. 
We choose to use $\mathbf{Int}$ instead of $\mathbf{Open}(\mathbb{R})$ due to technical issues that arise when we begin smoothing the functors. 
Basically, smoothing the functor does not produce a cosheaf when the intervals are replaced by arbitrary open sets in $\mathbb{R}$.
Consider the example of Fig.~\ref{fig:counterExample}, where  $\mathbb{X}$ is a line with map $f$ projection onto $\mathbb{R}$. Say $U^\e$ is the thickening of a set, $U^\e = \{x \in \mathbb{R} \mid |x-U| < \e\}$. 
Then we can pick an $\e$ so that $A^\e$ is two disjoint intervals, and $(A \cup B)^\e$ is one interval. 
Let $F$ be the functor $U \mapsto \pi_0 f^{-1}(U)$ which is a cosheaf representing the Reeb graph. 
Then the functor $F \circ (\cdot)^\e$ is not a cosheaf since by the diagram,}
\begin{equation*}
    \begin{tikzcd}
     \emptyset = F ( (A\cap B)^\e ) \ar[r] \ar[d] 
     & F(A^\e) = \{ \bullet  \bullet \} \ar[d, dashed] \\
     \{ \bullet  \} = F(B^\e) \ar[r, dashed] & \mathrm{colim} = \{ \bullet \bullet \bullet\}
    \end{tikzcd}
\end{equation*}
\myblue{$F(A\cup B)^\e = \{ \bullet\}$ is not the colimit of $F(A^\e)$ and $F(B^\e)$.\footnote{We thank Vin de Silva for this counterexample.}} 
\end{remark}

\subsection{Categorified mapper}
\label{sec:CategorifiedMapper}

In this section, we interpret classic mapper (for scalar functions), a topological descriptor, as a category theoretic object. This interpretation, in terms of cosheaves and category theory, simplifies many of the arguments used to prove convergence results in Section \ref{sec:MainResults}.
We first review the classic mapper  and then discuss the categorified mapper. 
The main ingredient needed to define the mapper construction is a choice of cover.
We say a cover of $\R$ is \emph{good} if all intersections are contractible.
A cover $\cU$ is \emph{locally finite} if for every $x \in \R$, $\cU_x=\{V\in \cU:x\in V\}$ is a finite set. 
In particular, locally finiteness implies that the cover restricted to a compact set is finite.
For the remainder of the paper, we work with \emph{nice covers} which are  good,  locally finite, and consist only of connected intervals, see Figure~\ref{fig:enhanced-mapper}(c) for an example. 

We will now introduce a categorification of mapper. Let $\mathcal{U}$ be a nice cover of $\R$. 
 Let $\mathcal{N}_\mathcal{U}$ be the nerve of $\cU$, endowed with the Alexandroff topology. Consider the continuous map
\begin{eqnarray*}
\eta:\R&\rightarrow& \mathcal{N}_\mathcal{U}\\
x&\mapsto& \bigcap_{V\in \cU_x} V,
\end{eqnarray*}
where the intersection $\bigcap_{V\in \cU_x} V$ is viewed as an open simplex of $\cN_\cU$. The \emph{mapper functor} $\cM_\cU:\Set^\Int\rightarrow \Set^\Int$ can be defined as
\[
\cM_\cU(\sC)= \eta^\ast(\eta_\ast(\sC)),
\]
where $\eta^\ast$ and $\eta_\ast$ are the (pre)-cosheaf-theoretic pull-back and push-forward operations respectively. However, rather than defining $\eta^\ast$ and $\eta_\ast$ in generality, we choose to work with an explicit description of $\cM_\cU(\sC)$ given below. For notational convenience, define 
\begin{eqnarray*}
\cI_\cU:\Int&\rightarrow&\Int\\
U&\mapsto& \eta^{-1}(\St( \eta(U))),
\end{eqnarray*}
where $\St (\eta(U))$ denotes the minimal open set in $\cN_\cU$ containing \myblue{$\eta(U):=\cup_{x\in U}\eta(x)$} (the open star of $\eta(U)$ in $\cN_\cU$). It is often convenient to identify $\cI_\cU(U)$ with a union of open intervals in $\R$.
\begin{lemma}
Using the notation defined above, we have the equality
\[\cI_\cU(U)=\bigcup_{x\in U}\bigcap_{V\in\cU_x}V,\]
where $\bigcap_{V\in\cU_x}V$ is viewed as a subset of $\R$ (not as a simplex of $\cN_\cU$).
\end{lemma}
\begin{proof}
If $y\in \bigcup_{x\in U}\bigcap_{V\in\cU_x}V$, then there exists an $x\in U$ such that $y\in V$ for all $V\in \cU_x$. In other words, $\cU_x\subseteq \cU_y$. Therefore, $\eta(y)\ge \eta(x)$ in the partial order of $\cN_\cU$. Therefore, $\eta(y)\in \St (\eta(U))$. This implies that $\bigcup_{x\in U}\bigcap_{V\in\cU_x}V \subseteq \cI_\cU(U)$. For the reverse inclusion, assume that $u\in \cI_\cU(U)$, i.e., $\eta(u)\in\St(\eta(U))$. This implies that there exists $v\in U$ such that $\eta(u)\ge \eta(v)$. In other words, $\cU_v\subseteq \cU_u$. Therefore $u\in \cap_{V\in \cU_v}V$, and $u\in \bigcup_{v\in U}\bigcap_{V\in\cU_v}V$.
\qed
\end{proof}
 Under this identification, it is clear that $\cI_\cU(U)$ is an open set in $\R$ (since the open cover $\cU$ is locally finite), and if $U\subset V$ then $\cI_\cU(U)\subset \cI_\cU(V)$. Moreover, since $\bigcap_{V\in\cU_x}V$ is an interval open neighborhood of $x$ and $U$ is an open interval, then $\cI_\cU(U)$ is an open interval. Therefore, $\cI_\cU$ can be viewed as a functor from $\Int$ to $\Int$. 

Finally, we can give $\cM_\cU(\sC)$ an explicit description in terms of the functor $\cI_\cU$. 
\begin{definition}
The \emph{mapper functor} $\cM_\cU:\Set^\Int\rightarrow \Set^\Int$ is defined by
\[
\cM_\cU(\sC)(U):= \sC(\cI_\cU(U)),
\]
for each open interval $U\in\Int$. 
\end{definition}

Since $\cI_\cU$ is a functor from $\Int$ to $\Int$, it follows that $\cM_\cU$ is a functor from $\Set^\Int$ to $\Set^\Int$. Hence, $\cM_\cU(\sC)$ is a functor from the category of pre-cosheaves to the category of pre-cosheaves. In the following proposition, we show that if $\sC$ is a cosheaf, then $\cM_\cU(\sC)$ is in fact a constructible cosheaf. 

\begin{proposition}
  \label{prop:mapper-functor}
Let $\cU$ be a finite nice open cover of $\R$. The mapper functor $\cM_\cU$ is a functor from the category of cosheaves on $\R$ to the category of constructible cosheaves on $\R$:
 \[
 \cM_\cU:\text{\bf{CSh}}\rightarrow \text{\bf{CSh}}^\text{\bf{c}}.
 \]
 Moreover, the set of critical points of $\cM_\cU(\sF)$ is a subset of the set of boundary points of open sets in $\cU$.
\end{proposition}

\begin{proof}
We will first show that if $\sC$ is a cosheaf on $\R$, then $\cM_\cU(\sC)$ is a cosheaf on $\R$. We have already shown that $\cM_\cU(\sC)$ is a pre-cosheaf. So all that remains is to prove the colimit property of cosheaves. Let $U\in\Int$ and $\cV\subset \Int$ be a cover of $U$ by open intervals which is closed under intersections. By definition of $\cM_\cU(\sC)$, we have
\[
\varinjlim_{V\in\cV}\cM_\cU(\sC)(V) = \varinjlim_{V\in\cV} \sC(\cI_\cU(V)).
\]
Notice that $\cI_\cU(\cV):=\{\cI_\cU(V):V\in\cV\}$ forms an open cover of $\cI_\cU(U)$. However, in general this cover is no longer closed under intersections. We will proceed by showing that passing from $\cV$ to $\cI_\cU(\cV)':=\{\bigcap_{i\in I} W_i:\{W_i\}_{i\in I}\subset \cI_\cU(\cV)\}$ does not change the colimit
\[
\varinjlim_{V\in\cV} \sC(\cI_\cU(V)).
\]
Suppose $I_1$ and $I_2$ are two open intervals in $\cV$ such that $I_1\cap I_2=\emptyset$ and $\cI_\cU(I_1)\cap \cI_\cU(I_2)\neq \emptyset$. Recall that $\cU'$ is the union of $\cU$ with all intersections of cover elements in $\cU$, i.e., the closure of $\cU$ under intersections. By the identification
\[\cI_\cU(I_i)=\bigcup_{x\in I_i}\bigcap_{V\in\cU_x}V,\]
there exists a subset $\{W_j\}_{j\in J}\subset \cU'$ such that \[\cI_\cU(I_1)\cap\cI_\cU(I_2)=\bigcup_{j\in J}W_j.\]
Suppose there exist $V_1,V_2\in\cU'$ such that $V_i\subsetneq V_1 \cup V_2$ (i.e., one set is not a subset of the other), and $V_1\cup V_2\subset \cI_\cU(I_1)\cap\cI_\cU(I_2)$. In other words, suppose that the cardinality of $J$, for any suitable choice of indexing set, is strictly greater than 1. Then there exists $x_1,x_2\in I_1$ such that $x_1\in V_1\setminus V_2$ and $x_2\in V_2\setminus V_1$. Let $w$ either be a point contained in $V_1\cap V_2$ (if $V_1\cap V_2\neq \emptyset$) or a point which lies between $V_1$ and $V_2$. Since $I_1$ is connected, we have that $w\in I_1$. A similar argument shows that $w\in I_2$, which implies the contradiction $I_1\cap I_2\neq\emptyset$. Therefore, \[\cI_\cU(I_1)\cap\cI_\cU(I_2) = W,\] for some $W\in\cU'$.
Suppose $W=\bigcap_{k\in K}W_k$ for some $\{W_k\}_{k\in K}\subset\cU$, and let $I_1=J_1, J_2, \cdots, J_n=I_2$ be a chain of open intervals in $\cV$, such that $J_j\cap J_{j+1}\neq \emptyset$. We have that
\[
I_1\cup \bigcup_{k\in K}W_k\cup I_2
\]
is connected, because $I_1$, $I_2$, and $\bigcup_{k\in K}W_k$ are intervals with $\bigcup_{k\in K}W_k\cap I_1$ and $\bigcup_{k\in K}W_k\cap I_2$ nonempty. Therefore, for each $j$, $J_j\cap W_k\neq \emptyset $ for some $k$, i.e., $W\subset \cI_\cU(J_j)$. In conclusion, we have shown that
\[
\cI_\cU(I_1)\cap\cI_\cU(I_2)\subseteq \cI_\cU(J_j)\text{ for each $j$.}
\]
Following the arguments in the proof of Proposition 4.17 of \citep{deSilvaMunchPatel2016}, it can be shown that
\[
\varinjlim_{V\in\cV} \sC(\cI_\cU(V)) =\varinjlim_{U\in\cI_\cU(\cV)} \sC(U)= \varinjlim_{U\in\cI_\cU(\cV)'} \sC(U).
\]
Since $\sC$ is a cosheaf, we can use the colimit property of cosheaves to get
\[
\varinjlim_{V\in\cV}\cM_\cU(\sC)(V) = \sC(\cI_\cU(U)).
\]
Therefore $\cM_\cU(\sC)$ is cosheaf. We will proceed to show that $\cM_\cU(\sC)$ is constructible.

Let $S$ be the set of boundary points for open sets in $\cU$. Since $\cU$ is a finite, good cover of $\R$, $S$ is a finite set. If $U\subset V$ are two open sets in $\R$ such that $U\cap S = V\cap S$, then $\cI_\cU(U)=\cI_\cU(V)$. Therefore $\cM_\cU(\sF)(U)\rightarrow\cM_\cU(\sF)(V)$ is an isomorphism.
\qed
\end{proof}


We use the mapper functor to relate Reeb graphs (the display locale of the Reeb cosheaf $\sR_f$) to the enhanced mapper graph (the display locale of $\cM_\cU(\sR_f)$).
In particular, the error is controlled by the resolution of the cover, as defined below.
\begin{definition}\label{def:cover}
Let $\cU$ be a nice cover of $\R$ and $\sF$ a cosheaf on $\R$. The \emph{resolution} of $\cU$ relative to $\sF$, denoted $\res_\sF\cU$, is defined to be the maximum of the set of diameters of $\cU_\sF:=\{V\in\cU:\sF(V)\neq\emptyset\}$:
\[
\res_\sF\cU:=\max\{\diam(V):V\in\cU_\sF\}.
\]
\end{definition}
Here we understand the diameter of open sets of the form $(a,+\infty)$ or $(-\infty,b)$ to be infinite. Therefore, the resolution $\res_\sF\cU$ can take values in the extended non-negative numbers $\R_{\ge0}\sqcup \{+\infty\}$. 
\myblue{
\begin{remark}
If $\sR_f$ is a Reeb cosheaf of a constructible $\R$-space $(\X,f)$, then $\sR_f(V)\neq\emptyset $ if and only if $V\cap f(\X)\neq \emptyset$. 
\end{remark}
\begin{definition}\label{def:resolution} Define $\res_f\cU$ by 
\[
\res_{f}\cU:=\max\{\diam(V):V\in\cU_f\},
\]
where $\cU_f:=\{V\in\cU: V\cap f(\X)\neq \emptyset\}$.
\end{definition}
The following theorem is analogous to \cite[Theorem 1]{MunchWang2016}, adapted to the current setting. Specifically, our definition of the mapper functor $\cM_\cU$ differs from the functor $\mathcal{P}_K$ of \citep{MunchWang2016}, and the convergence result of \citep{MunchWang2016} is proved for multiparameter mapper (whereas the following result is only proved for the one-dimensional case). } 
\begin{theorem}[{cf. \cite[Theorem 1]{MunchWang2016}}]
  \label{thm:interleave}
  Let $\cU$ be a nice cover of $\R$, and $\sF$ a cosheaf on $\R$. Then
$$
d_I(\sF,\cM_\cU(\sF))\le \res_\sF\cU.
$$
\end{theorem}
\begin{proof}
If $\res_\sF\cU=+\infty$, then the inequality is automatically satisfied. Therefore, we will work with the assumption that $\res_\sF\cU<+\infty$. Let $\delta_\cU=\res_\sF\cU<+\infty$. We will prove the theorem by constructing a $\delta_\cU$-interleaving of the sheaves $\sF$ and $\cM_\cU(\sF)$. Suppose $I\in\Int$. For each $x\in I$, let $W_x=\bigcap_{V\in\cU_x}V$. Recall that 
\[\cI_\cU(I)=\bigcup_{x\in I}W_x. \]
 Ideally, we would construct an interleaving based on an inclusion of the form $\cI_\cU(I)\subset I_{\delta_\cU}$. However, this inclusion will not always hold. For example, if $\cU$ is a finite cover, then it is possible for $I$ to be a bounded open interval, and for $\cI_\cU(I)$ to be unbounded. 
 
 \myblue{We will include a simple example to illustrate this behavior. Suppose $\cU = \{(-\infty, -1),(-2,2), (1,+\infty)\}$ and let $\sF$ be the constant cosheaf supported at $0$, i.e. $\sF(U)=\emptyset$ if $0\notin U$ and $\sF(V) = \{\ast\}$ if $0\in V$. Consider the interval $I = (0,3)$. For each $x\in (0,1]\subset I$, we have that $W_x = (-2,2)$. If $x\in (1,2)\subset I$, then $W_x = (-2,2)\cap (1,+\infty)$. Finally, if $x\in [2,3)\subset I$, then $W_x = (1,+\infty)$. Therefore, $\cI_\cU(I) = (-2,+\infty)$, which is unbounded. However, we observe that $\sF((-\infty,-1))=\emptyset$, $\sF((-2,2))=\{\ast\}$, and $\sF((1,+\infty))=\emptyset$. Therefore, (in the notation of Definition \ref{def:cover}) $\cU_\sF=\{(-2,2)\}$, and $\res_\sF\cU = \diam((-2,2))=4$. }
 
 Although $\cI_\cU(I)$ may be unbounded, we can construct an interval $I'$ which is contained in $I_{\delta_\cU}$ and satisfies the equality $\sF(\cI_\cU(I))= \sF(I')$. The remainder of the proof will be dedicated to constructing such an interval. 
 
 Let $\cW:=\{U: U =\cap_{a\in A} W_a\text{ for some }A\subset I\}$ be an open cover of $\cI_\cU(I)$ which is closed under intersections and generated by the open sets $W_x$. Then the colimit property of cosheaves gives us the equality
\[\sF(\cI_\cU(I)) = \varinjlim_{U\in\cW} \sF(U). \]
Let $E:=\{e\in I:\sF(W_e) = \emptyset\} $. If $U = \cap_{a\in A }W_a$ and $A\cap E\neq\emptyset$, then $\sF(U)=\emptyset$. Let $\cW_{I\setminus E}=\{U\in \cW: U= \cap_{a\in A} W_a\text{ for some } A\subset I\setminus E\}$. We should remark on a small technical matter concerning $ I\setminus E$. In general, this set is not necessarily connected. If that is the case, we should replace $I\setminus E$ with the minimal interval which covers $I\setminus E$. Going forward, we will assume that $I\setminus E $ is connected. Altogether we have 
\[ \cM_\cU(\sF)(I)=\sF(\cI_\cU(I)) =\varinjlim_{U\in\cW} \sF(U)=\varinjlim_{U\in\cW_{I\setminus E}} \sF(U)=\sF\left(\bigcup_{x\in I\setminus E }W_x\right) . \]
If $x\in I\setminus E$, then $W_x\cap I\neq\emptyset$ and $\sF(W_x)\neq \emptyset$. Therefore, $W_x \subseteq I_{\delta_\cU}$, since $\diam(W_x)\le\delta_\cU$. 
Moreover, 
\[
 \bigcup_{x\in I\setminus E}W_x \subseteq I_{\delta_\cU} .
\]
The above inclusion induces the following map of sets
\[
\varphi_I:\cM_\cU(\sF)(I)\rightarrow \sF(I_{\delta_\cU}),
\]
which gives the first family of maps of the $\delta_\cU$-interleaving. The second family of maps
\[
\psi_I:\sF(I)\rightarrow \cM_\cU(\sF)(I_{\delta_\cU}),
\]
follows from the more obvious inclusion $I \subset \cI_\cU(I_{\delta_\cU})$. Since the interleaving maps are defined by inclusions of intervals, it is clear that the composition formulae are satisfied:
\[
\psi_{I_{\delta_\cU}}\circ\varphi_I=\sF[I\subset I_{2\delta_\cU}],\qquad \varphi_{I_{\delta_\cU}}\circ\psi_I=\cM_\cU(\sF)\left[I\subset I_{2\delta_\cU}\right].
\]
\qed
\end{proof}
\begin{remark} 
One might think that Theorem \ref{thm:interleave} can be used to obtain a convergence result for the mapper graph of a general $\R$-space. However, we should emphasize that the interleaving distance is only an extended \emph{pseudo}-metric on the category of all cosheaves. Therefore, even if the interleaving distance between $\sF$ and $\cM_\cU(\sF)$ goes to 0, this does not imply that the cosheaves are isomorphic. We only obtain a convergence result when restricting to the subcategory of constructible cosheaves, where the interleaving distance gives an extended metric. 
\end{remark}

The display locale $\mathfrak{D}(\cM_\cU(\sR_f))$ of the mapper cosheaf is a 1-dimensional CW-complex obtained by gluing the boundary points of a finite disjoint union of closed intervals, see Figure~\ref{fig:enhanced-mapper}(h). We will refer to this CW-complex as the \emph{enhanced mapper graph} of $(\X,f)$ relative to $\cU$, see Figure~\ref{fig:enhanced-mapper}(g). There is a natural surjection from $\mathfrak{D}(\cM_\cU(\sR_f))$ to the nerve of the connected cover pull-back of $\cU$, $\cN_{f^\ast(\cU)}$, i.e., from the enhanced mapper graph to the mapper graph, when the cover $\cU$ contains open sets with empty triple intersections.

Using the Reeb interleaving distance and the enhanced mapper graph, we obtain and reinterpret the main result of \citep{MunchWang2016} in the following corollary. 
\begin{corollary}[{cf. \cite[Corollary 6]{MunchWang2016}}]
  Let $\cU$ be a nice cover of $\R$, and $(\X,f) \in \Rspacec$. Then
$$
d_R(\cR(\X,f),\mathfrak{D}(\cM_\cU(\sR_f)))\le \res_f\mathcal{U}.
$$
\end{corollary}

Throughout this section we introduce several categories and functors which we will now summarize. Let $\Rgraph$ be the category of $\R$-graphs (i.e., Reeb graphs), $\Rspacec$ the category of constructible $\R$-spaces, $\Cshc$ be the category of constructible cosheaves on $\R$, $\cS_\e$ and $\cT_\e$ the smoothing and thickening functors, $\mathfrak{D}$ the display locale functor, and $\cM_\cU$ the mapper functor. Altogether, we have the following diagram of functors and categories, 
\begin{center}
\begin{tikzpicture}[scale=1]
\node (A) at (-2,0) {$\Rgraph$};
\node (B) at (2,0) {$\Cshc$};
\node (C) at (-2,-2) {$\Rspacec$};
\path[->]
(A) edge [bend right] node[left] {$\cT_\varepsilon$} (C)
    edge [loop left] node[left] {$\cS_\varepsilon$} (A);
\path[->] (C) edge [bend right] node[below] {$\cC$} (B)
(C) edge [bend right] node[right] {$\cR$} (A);
\path[->]
(B) edge [loop right] node[right] {$\cM_\cU$.} (B)
edge [bend right] node[below] {$\mathfrak{D}$} (A);
\end{tikzpicture}

\end{center}

\para{Enhanced mapper graph algorithm.} Finally, we briefly describe an algorithm for constructing the enhanced mapper graph, following the example in Figure \ref{fig:enhanced-mapper}. Let $(\X,f)$ be a constructible $\R$-space (see Section \ref{sec:R_Spaces}). For simplicity, suppose that the cover $\cU$ consists of open intervals, and contains no nonempty triple intersections ($U\cap V\cap W=\emptyset$ for all $U,V,W\in\cU$). Let $\R_0$ be the union of boundary points of cover elements in the open cover $\cU$. Let $\R_1$ be the complement of $\R_0$ in $\R$. The set $\R_0$ is illustrated with gray dots in Figure \ref{fig:enhanced-mapper}(e). We begin by forming the disjoint union of closed intervals,  
\[\coprod_I \overline{I}\times \pi_0(f^{-1}(U_I)),\]
where the disjoint union is taken over all connected components $I$ of $\R_1$, $\overline{I}$ denotes the closure of the open interval $I$, and $U_I$ denotes the smallest open set in $\cU\cup \{U\cap V \mid U,V\in \cU\}$ which contains $I$. In other words, $U_I$ is either the intersection of two cover elements in $\cU$ or $U_I$ is equal to a cover element in $\cU$.  The sets $\pi_0(f^{-1}(U_I))$ are illustrated in Figure \ref{fig:enhanced-mapper}(d). Notice that there is a natural projection map from the disjoint union to $\R$, given by projecting each point $(y,a)$ in the disjoint union onto the first factor, $y\in\R$. The enhanced mapper graph is a quotient of the above disjoint union by an equivalence relation on endpoints of intervals. This equivalence relation is defined as follows. Let $(y,a) \in \overline{I}\times \pi_0(f^{-1}(U_I)) $ and $(z,b) \in \overline{J}\times \pi_0(f^{-1}(U_J))$ be two elements of the above disjoint union. If $y\in\R_0$, then $y$ is contained in exactly one cover element in $\cU$, denoted by $U_y$. Moreover,if $y\in\R_0$, then there is a map $\pi_0(f^{-1}(U_I))\rightarrow \pi_0(f^{-1}(U_y))$ induced by the inclusion $U_I\subseteq U_y$. Denote this map by $\psi_{(y,I)}$. An analogous map can be constructed for $(z,b)$, if $z\in \R_0$. We say that $(y,a)\sim(z,b)$ if two conditions hold: $y=z$ is contained in $\R_0$, and $\psi_{(y,I)}(a)=\psi_{(z,J)}(b)$. The \emph{enhanced mapper graph} is the quotient of the disjoint union by the equivalence relation described above. 

For example, as illustrated in Figure~\ref{fig:enhanced-mapper}, seven cover elements of $\mathcal{U}$ in (c) give rise to a stratification of $\R$ into a set of points $\R_0$ and a set of intervals $\R_1$ in (e). 
For each interval $I$ in $\R_1$, we look at the set of connected components in $f^{-1}(U_I)$. 
We then construct disjoint unions of closed intervals based on the cardinality of $\pi_0(f^{-1}(U_I))$ for each $I \in \R_1$.
For adjacent intervals $I_1$ and $I_2$ in $\R_1$, suppose that $I_1$ is contained in the cover element $V$ and $I_2$ is equal to the intersection of cover elements $V$ and $W$ in $\cU$. We consider the mapping from $\pi_0(f^{-1}(U_{I_2}))$ to $\pi_0(f^{-1}(U_{I_1}))$ (d). Here, we have that $U_{I_2}=V\cap W$ and $U_{I_1}= V$. 
We then glue these closed intervals following the above mapping, which gives rise to the enhanced mapper graph (g). 
\myblue{Appendix~\ref{sec:pseudocode} outlines these algorithmic details in the form of pseudocode.}


\section{Model}
\label{sec:Model}

Let $\X$ be a compact locally path connected subset of $\R^d$. As stated in the introduction, study related to topological inference usually splits between noiseless and noisy settings.
In the former, we assume that a given sample is drawn from $\X$ directly, while in the latter we allow random perturbations that produce samples in $\R^d$ that need not be on $\X$, but rather in its vicinity.
In this paper, we address the noisy setting directly, using the machinery for super-level sets estimation developed in~\citep{BobrowskiMukherjeeTaylor2017}.
The basic inputs are a continuous function $f:\R^d\rightarrow \R$, and a probability density function $p:\R^{d}\rightarrow \R$.
Our $\R$-space of interest will be $(\X, f\vert_\X)$, and we will assume we are provided samples of $\X$ via $p$.
Then, given a nice cover $\cU$, we can compare the Reeb graph of $(\X, f\vert_\X)$ to the mapper graph computed from the samples. 

\subsection{Setup}
\label{ssec:Model_Basics}

In this section, we give some basic assumptions on $f$, $p$, $\cU$, and their interactions.
We start with some notation for the various sets of interest.
Let $\X_\delta=\{y\in\R^d:\inf_{x\in \X}d(x,y)\le \delta\}$ be the $\delta$-thickening of $\X$, and let $D_L=p^{-1}([L,+\infty))$ be a super-level set of $p$.
Given an open set $V\subset \R$, define $\X^V:=\X\cap f^{-1}(V)$.
Let $\X_\delta^V := \X_\delta\cap f^{-1}(V)$ be the elements of $\X_\delta$ which map to $V$, and $D^V_L:=D_L\cap f^{-1}(V)$.
See \ref{fig:AnnulusExample_Notation} for an example of this notation.
It is important to note that $\X_\delta^V$ is not necessarily equal to the $\delta$-thickening of $\X^V$.

\begin{figure}
    \centering
    \includegraphics[width = .4\textwidth]{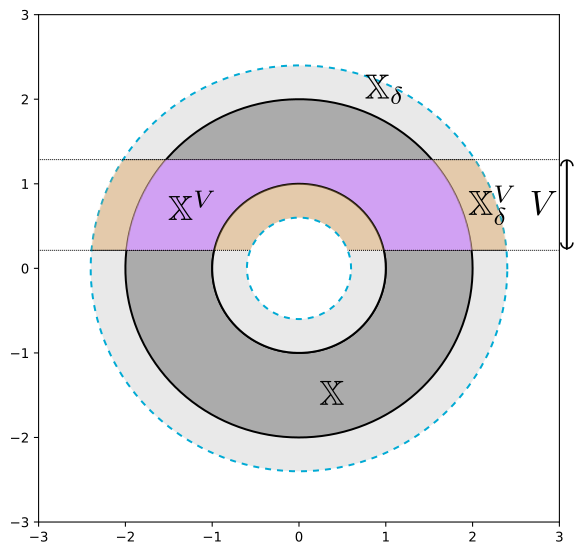}
    \caption{This figure illustrates the inverse images $\X^V$ (in purple) and $\X_\delta^V$ (union of tan and purple) for an annulus with height function and open interval $V$. Notice that in this example the $\delta$-thickening of $\X^V$ would include $\X_\delta^V$ as a proper subset, hence $(\X^V)_\delta\neq \X_\delta^V$. }
    \label{fig:AnnulusExample_Notation}
\end{figure}

With this notation, we will assume that $p$ is $\epsilon$-concentrated on $\X$ as defined next with an example given in  \ref{fig:AnnulusExample_Concetrated}.
\begin{definition}
  \label{defn:epsilon-concentrated}
A probability density function $p$ is \emph{$\epsilon$-concentrated on $\X$} if there exists open intervals $I_1,I_2$, and a real number $\delta>0$ such that
\[
\X\subset D_{L_1}\subset \X_\delta\subset D_{L_2}\subset \X_\epsilon,
\]
for any $L_1\in I_1$ and $L_2\in I_2$. 
\end{definition}
\begin{definition}
  \label{defn:concentrated}
A probability density function $p$ is \emph{concentrated on $\X$} if $p$ is $\epsilon$-concentrated on $\X$ for all $\epsilon>0$.
\end{definition}

\begin{figure}
    \centering
    \includegraphics[width = .3\textwidth]{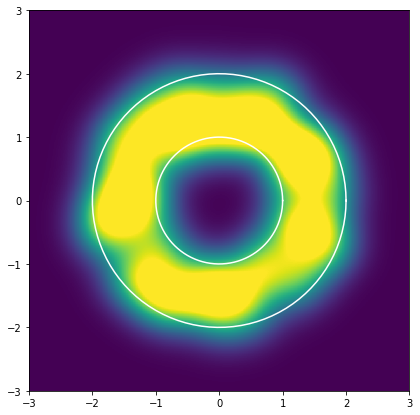}
    \includegraphics[width = .3\textwidth]{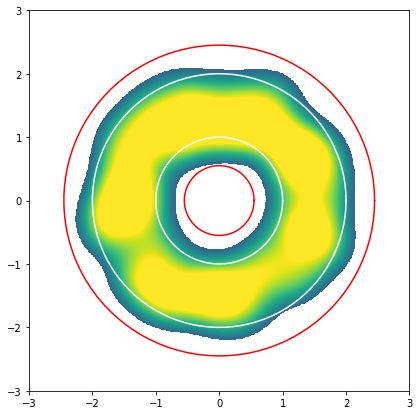}
    \includegraphics[width = .3\textwidth]{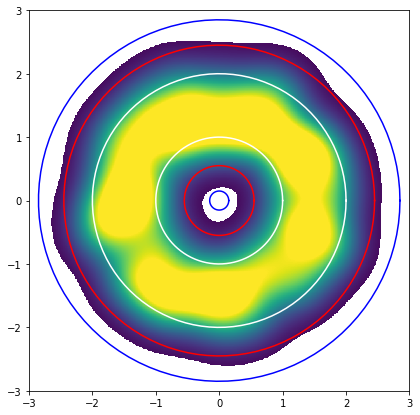}
    \caption{An example of the concentrated definition. The left side of the figure illustrates a probability density function (PDF) 
 $p$ which is $\epsilon$-concentrated on an annulus $\X$. The center image illustrates the thickened space $\X_\delta$, bounded by the red curves, and the super-level set $D_{L_1}$. The right side of the figure illustrates the thickened space $\X_\epsilon$, bounded by the blue curves, and the super-level set $ D_{L_2}$. Together, we see that $\X\subset D_{L_1}\subset \X_\delta\subset D_{L_2}\subset \X_\epsilon$. 
    }
    \label{fig:AnnulusExample_Concetrated}
\end{figure}

We now turn our attention to $\cU$, a nice cover of $\R$.

\begin{definition}
The \emph{local $H_0$-critical value over $V$} is defined as
$$ \delta_V=\sup\{\delta \mid  H_0(\X^V)\xrightarrow{\cong}H_0(\X^V_\delta)\}.$$

Let
$
\cU' :=\{V\subset \R: V=\cap_{\alpha\in A}U_\alpha \textrm{ for some } \{U_\alpha\}_{\alpha\in A}\subset \cU\}.
$
The \emph{global $H_0$-critical value over $\cU$} is defined as
\[\delta_\cU:=\min_{V\in\cU'}\delta_V.\]
\end{definition}

Throughout the paper, we assume that the global $H_0$-critical value over $\cU$ is positive, i.e.
   \[\delta_\cU=\min_{V\in\cU'}\delta_V>0.\]
The positivity of local $H_0$-critical values is nontrivial and often fails for constructible $\R$-spaces which have singularities which lie over the boundary of one of the open sets in the open cover $\cU$. In future work, it would be interesting to relax this assumption, and study convergence when the diameter of the union of open sets $V$ for which $\delta_V=0$, is small.

\subsection{Approximation by super-level sets}
\label{ssec:ApproximationBySuperLevelSets}
In this section, we study how super-level sets of probability density functions (PDFs) can model the topology of constructible $\R$-spaces.

We need some further control over the relationship between the PDF $p$ and the cover elements via the following definition.
\begin{definition}
  \label{defn:H0_Regular}
Given an open set $V$, we say that $L$ is \emph{$H_0$-regular over $V$} if there exists $\nu>0$ such that for all $\e_1<\e_2\in(L-\nu,L+\nu)$, the inclusion $D_{\e_2}\subset D_{\e_1}$ induces an isomorphism  $H_0(D^V_{\e_2})\xrightarrow{\cong}H_0(D^V_{\e_1})$.
\end{definition}
Throughout the paper we will assume that the PDF $p$ is \emph{tame}, in the sense that the set of points which are $H_0$-regular over $V$ is dense in $\R$, for any given open set $V$. 

Assume the global $H_0$-critical value $\delta_\cU$ is positive, and $p$ is $\delta_2$-concentrated on $\X$ for some $\delta_2$ such that $0<\delta_2<\delta_\cU$. By definition, there exist $L_1,L_2$ and $\delta_1$ such that 
    \begin{enumerate}\denselist
        \item $\X\subset D_{L_1}\subset \X_{\delta_1}\subset D_{L_2}\subset \X_{\delta_2}$
        \item $0<\delta_1<\delta_2<\delta_{\cU}$
        \item $L_1$ and $L_2$ are $H_0$-regular over $V$ for each $V\in\cU'$.
    \end{enumerate}

The set of points which are $H_0$-regular over $V$ for each $V\in \cU'$ is dense in $\R$. If $L_1$ is not $H_0$-regular over $V$ for some $V\in\cU'$, then $L_1$ can be turned into a regular value with an arbitrarily small perturbation. Moreover, by the Definition \ref{defn:epsilon-concentrated}, this perturbation can be done without breaking the chain of inclusions $\X\subset D_{L_1}\subset \X_{\delta_1}\subset D_{L_2}\subset\X_{\delta_2}$. We therefore continue under the assumption that $L_1$ is $H_0$-regular over $V$ for each $V\in\cU'$.

\begin{proposition}\label{prop:box}
Assume that $p$ is $\epsilon$-concentrated on $\X$ for some $\epsilon<\delta_\cU$. Let
\[
\sD(V):= \im \left(H_0(D_{L_1}^V)\rightarrow H_0(D_{L_2}^V)\right).
\]
Then for each $V \in \cU'$, we have $H_0(\X^V) \cong \sD(V) $
and further for each $V\subset W\in \cU'$ the following diagram commutes, 
\begin{equation*}
  \begin{tikzcd}
    H_0(\X^V) \ar[r,"\cong"] \ar[d] & \sD(V) \ar[d] \\
    H_0(\X^U) \ar[r,"\cong"] & \sD(U).
  \end{tikzcd}
\end{equation*}
\end{proposition}

The proof will require the following technical lemma.
\begin{lemma}
\label{lemma:interleave}
Suppose we have the following commutative diagram of vector spaces
      \begin{center}
\begin{tikzpicture}[scale=1]
\node (A) at (-2,1) {$A$};
\node (B) at (2,1) {$B$};
\node (C) at (0,0) {$D$};
\node (H) at (-4,0) {$C$};
\node (J) at (4,0) {$E$};
\path[->] (H) edge node[above] {$\cong$}(C);
\path[->] (C) edge node[above] {$\cong$} (J);
\path[->] (A) edge (B);
\path[->] (A) edge (C);
\path[->] (C) edge (B);
\path[->] (H) edge (A);
\path[->] (B) edge (J);
\end{tikzpicture}
\end{center}
with $C\cong D\cong E$. Then $\im(D\rightarrow B)= \im(A\rightarrow B)$ and the map
\[
D\xrightarrow{\cong} \im(A\rightarrow B)
\]
is an isomorphism of vector spaces.
\end{lemma}
\begin{proof}
The map $D\rightarrow B$ is injective since $D\rightarrow E$ is an isomorphism and the diagram commutes. Therefore, $\im(D\rightarrow B)\cong D$. Moreover, since the diagram commutes, $\im(A\rightarrow B)\subset \im(D\rightarrow B)$. Suppose $b\in \im(D\rightarrow B)$, i.e., there exists $d\in D$ which maps to $b$. Since $C\rightarrow D$ is an isomorphism, there exists $c\in C$ which maps to $d\in D$. Let $a\in A$ be the image of $c\in C$ under the map $C\rightarrow A$. Since the diagram commutes, we have that $a\in A$ maps to $b\in B$ under the map $A\rightarrow B$. Therefore, $b\in\im(A\rightarrow B)$. We have therefore shown that $\im(A\rightarrow B)=\im(D\rightarrow B)\cong D$.
\qed
\end{proof}

\begin{proof}[Proof of \ref{prop:box}]
Choose $\delta_2>0$ such that $\delta_2<\delta_\cU$ and $p$ is $\delta_2$-concentrated on $\X$. Applying the definition of $\delta_2$-concentrated, we have $\X\subset D_{L_1}\subset \X_{\delta_1}\subset D_{L_2}\subset \X_{\delta_2}$. For $V\subset W$ we have the following commutative diagram of vector spaces
   \begin{center}
\begin{tikzpicture}[scale=1]
\node (A) at (-2,1) {$H_0(D_{L_1}^V)$};
\node (B) at (2,1) {$H_0(D^V_{L_2})$};
\node (C) at (0,0) {$H_0(\X^V_{\delta_1})$};
\node (D) at (-2,-3) {$H_0(D_{L_1}^W)$};
\node (E) at (2,-3) {$H_0(D_{L_2}^W)$};
\node (F) at (0,-2) {$H_0(\X^W_{\delta_1})$};
\node (G) at (-4,-2) {$H_0(\X^W)$};
\node (H) at (-4,0) {$H_0(\X^V)$};
\node (I) at (4,-2) {$H_0(\X^W_{\delta_2})$};
\node (J) at (4,0) {$H_0(\X^V_{\delta_2})$};
\path[->] (H) edge (C);
\path[->] (C) edge (J);
\path[->] (G) edge (F);
\path[->] (F) edge (I);
\path[->] (A) edge (B);
\path[->] (A) edge (D);
\path[->] (A) edge (C);
\path[->] (B) edge (E);
\path[->] (C) edge (B);
\path[->] (D) edge (E);
\path[->] (D) edge (F);
\path[->] (C) edge (F);
\path[->] (F) edge (E);
\path[->] (H) edge (A);
\path[->] (B) edge (J);
\path[->] (G) edge (D);
\path[->] (E) edge (I);
\path[->] (J) edge (I);
\path[->] (H) edge (G);

\end{tikzpicture}
\end{center}
Since $\delta_1 < \delta_2<\delta_\cU$, by definition of global $H_0$-critical value over $\cU$, all four horizontal maps
\[
H_0(\X^V)\longrightarrow H_0(\X_{\delta_1}^V)\longrightarrow H_0(\X_{\delta_2}^V)
\qquad
\text{and}
\qquad
H_0(\X^W)\longrightarrow H_0(\X_{\delta_1}^W)\longrightarrow H_0(\X_{\delta_2}^W)
\]
are isomorphisms. Applying \ref{lemma:interleave}, we can conclude that
\[
H_0(\X^V)\longrightarrow\sD(V)
\qquad
\text{and}
\qquad
H_0(\X^W)\longrightarrow\sD(W)
\]
are isomorphisms of vector spaces.
Since the diagram commutes, the image of $\sD(V)$ under the map $H_0(D_{L_2}^V)\rightarrow H_0(D_{L_2}^W)$ is contained in $\sD(W)$.
Therefore, $H_0(D_{L_2}^V)\rightarrow H_0(D_{L_2}^W)$ induces a map $\sD(V)\rightarrow \sD(W)$, which completes the commutative diagram of the theorem.
\qed
\end{proof}
\subsection{Point-cloud mapper algorithm}
\label{ssec:PointCloudMapper}
 Given data $\{X_1,\cdots, X_n\}\overset{\text{i.i.d}}{\sim}p$, where $p$ is a PDF, we can estimate $p$ using a kernel density estimator (KDE) of the form,
    \[\hat{p}(x):=\frac{1}{C_K n r^d}\sum_{i=1}^nK_r(x-X_i),\]
    where $K(x)$ is a given kernel function, $K_r := K(x/r)$, and $C_K$ is a constant defined below. The kernel function should satisfy the following:
\begin{enumerate}\denselist
    \item $\mathrm{supp}(K)\subset B_1(0)$, and $K(x)$ is smooth in $B_1(0)$.
    \item $K(x) \in [0,1]$, and $\max_x K(x) = K(0) = 1$,
    \item $\int_{\R^d} K(x)dx = C_K$ with $C_K\in (0,1)$.
\end{enumerate}
    Using $\hat p$ we can estimate the super-level sets $D_L$ using
\begin{equation}
\label{eqn:def_DL}
\hat{D}_L(n,r):=\bigcup_{i: \hat p(X_i) \ge L} B_r(X_i),
\end{equation}
and the sets $D_L^V$ using
\begin{equation}
\label{eqn:def_DLV}
\hat{D}^V_L(n,r):= \hat{D}_L(n,r)\cap f^{-1}(V).
\end{equation}
Choose $\varepsilon_i$ such that $L_i+2\varepsilon_i,L_i-2\varepsilon_i$ are within the $H_0$-regularity range of $L_i$ over $V$ for each $V\in\cU$ and $L_1-2\varepsilon_1 >L_2+2\varepsilon_2$. In the following, we will use the term ``with high probability" (w.h.p.) to mean  that the probability of an event to occur converges to $1$ as $n\to \infty$.
\begin{proposition}\label{BMT:thm}
Fix $L$ and $V$, and set $\hat D_L^V := \hat D_L^V(n,r)$. Fixing $\varepsilon>0$, there exists a constant $C_\varepsilon >0$ (independent of $L$ and $V$) such that if $nr^d \ge C_\varepsilon \log n$, then the following diagram of inclusion relations holds w.h.p., 
     \begin{center}
\begin{tikzpicture}[scale=1]
\node (A) at (-2,1) {$\hat{D}_{L+\varepsilon}^V$};
\node (B) at (2,1) {$\hat{D}_{L-\varepsilon}^V$};
\node (C) at (0,0) {$D_{L}^V$};
\node (H) at (-4,0) {$D_{L+2\varepsilon}^V$};
\node (J) at (4,0) {$D_{L-2\varepsilon}^V$};
\path[right hook->] (H) edge (C);
\path[right hook->] (C) edge (J);
\path[right hook->] (A) edge (B);
\path[right hook->] (A) edge (C);
\path[right hook->] (C) edge (B);
\path[right hook->] (H) edge (A);
\path[right hook->] (B) edge (J);
\end{tikzpicture}
\end{center}
\end{proposition}

\begin{proof}
The proof appears as part of the proof of Theorem 3.3 in {\citep{BobrowskiMukherjeeTaylor2017}}.
\qed
\end{proof}
Next, define the random vector space
\[ \hat{\sD}_i(V):= \im \left(H_0(\hat{D}_{L_i+\varepsilon_i}^V)\rightarrow H_0(\hat{D}_{L_i-\varepsilon_i}^V)\right).\]
\begin{corollary}\label{cor:isom}
If $nr^d \ge C_{\varepsilon_i}\log n$, then w.h.p. the random map
\[ H_0(D_{L_i}^V)\rightarrow \hat{\sD}_i(V)\]
is an isomorphism.
\end{corollary}
\begin{proof}
The corollary follows from applying \ref{lemma:interleave} to \ref{BMT:thm}.
\qed
\end{proof}
From here on, unless otherwise stated, we will assume that $r$ is chosen so that $nr^d \ge \max(C_{\varepsilon_1}, C_{\varepsilon_2})\log n$, so that \ref{cor:isom} holds for both $\varepsilon_1,\varepsilon_2$.

\begin{proposition}
  \label{prop:level1-box}
For every $V\subset W\in \cU'$,
we have the following commutative diagram w.h.p., 
\begin{center}
\begin{tikzpicture}[scale=1]
\node at (0,0) {$\circlearrowright$};
\node (A) at (-1,1) {$H_0(D_{L_i}^V)$};
\node (B) at (1,1) {$ \hat{\sD}_i(V)$};
\draw[<-] (B) edge node[above] {$\cong$} (A);
\node (C) at (-1,-1) {$H_0(D_{L_i}^W)$};
\node (D) at (1,-1) {$ \hat{\sD}_i(W)$.};
\path[<-] (D) edge node[above] {$\cong$} (C);
\path[->] (A) edge (C);
\path[->] (B) edge (D);
\end{tikzpicture}
\end{center}
\end{proposition}
\begin{proof}
 The proof follows the same arguments as the proof of Proposition \ref{prop:box}, and using Corollary \ref{cor:isom}.
 \qed
\end{proof}

    Finally, we define the following random vector space, $$\hat{\sD}(V) : = \im \left( H_0(\hat{D}^V_{L_1+\varepsilon_1})\rightarrow H_0(\hat{D}^V_{L_2-\varepsilon_2})\right). $$
\begin{proposition}\label{prop:level2-box}
Assume that $p$ is $\epsilon$-concentrated on $\X$ for some $\epsilon<\delta_\cU$. For every $V\subset W\in \cU'$, we have the following commutative diagram w.h.p., 
    \begin{center}
\begin{tikzpicture}[scale=1]
\node at (0,0) {$\circlearrowright$};
\node (A) at (-1,1) {$H_0(\X^V)$};
\node (B) at (1,1) {$\hat{\sD}(V) $};
\path[<-] (B) edge node[above] {$\cong$} (A);
\node (C) at (-1,-1) {$H_0(\X^W)$};
\node (D) at (1,-1) {$\hat{\sD}(W),$};
\path[<-] (D) edge node[above] {$\cong$} (C);
\path[->] (A) edge (C);
\path[->] (B) edge (D);
\end{tikzpicture}
\end{center}
where the constants $L_1$ and $L_2$ (defining $\hat{\sD}$) are given by the definition of $\epsilon$-concentrated, and the constants $\varepsilon_1$ and $\varepsilon_2$ are given by the $H_0$-regularity of $L_1$ and $L_2$, respectively. 
\end{proposition}
\begin{proof}
We will use the assumption that $L_1-2\varepsilon_1> L_2+2\varepsilon_2$ repeatedly for each of the super-level set inclusions in the proof. The inclusion of spaces $\hat{D}^V_{L_1-\varepsilon_1}\subset \hat{D}^V_{L_2-\varepsilon_2}$ induces a homomorphism $H_0(\hat{D}^V_{L_1-\varepsilon_1})\rightarrow H_0( \hat{D}^V_{L_2-\varepsilon_2})$. Restricting the domain of this map, we get a homomorphism $\hat{\sD_1}(V)\rightarrow  H_0( \hat{D}^V_{L_2-\varepsilon_2})$.  Since $L_1-\varepsilon_1>L_2+\varepsilon_2>L_2-\varepsilon_2$, the map $\hat{\sD_1}(V)\rightarrow  H_0( \hat{D}^V_{L_2-\varepsilon_2})$ factors through $H_0(\hat{D}_{L_2+\varepsilon_2}^V)\rightarrow H_0(\hat{D}_{L_2-\varepsilon_2}^V)$, forming the commutative diagram
 \begin{center}
\begin{tikzpicture}[scale=1]
\node at (-.25,.25) {$\circlearrowright$};
\node (A) at (-1,1) {$\hat{\sD_1}(V)$};
\node (B) at (2,1) {$ H_0( \hat{D}^V_{L_2-\varepsilon_2})$};
\path[<-] (B) edge node[above] {} (A);
\node (C) at (-1,-1) {$H_0(\hat{D}_{L_2+\varepsilon_2}^V)$};
\path[<-] (B) edge node[above] {} (C);
\path[->] (A) edge (C);
\end{tikzpicture}
\end{center}
This implies that $\im(\hat{\sD_1}(V)\rightarrow H_0(\hat{D}^V_{L_2-\varepsilon_2}))\subset \hat{\sD_2}(V) $, and gives a map $\hat{\sD_1}(V)\rightarrow \hat{\sD_2}(V)$, which w.h.p. completes the following commutative diagram,
    \begin{center}
\begin{tikzpicture}[scale=1]
\node at (0,0) {$\circlearrowright$};
\node (A) at (-1,1) {$H_0(D_{L_1}^V)$};
\node (B) at (1,1) {$\hat{\sD}_1(V) $};
\path[<-] (B) edge node[above] {$\cong$} (A);
\node (C) at (-1,-1) {$H_0(D_{L_2}^V)$};
\node (D) at (1,-1) {$\hat{\sD}_2(V) $};
\path[<-] (D) edge node[above] {$\cong$} (C);
\path[->] (A) edge (C);
\path[->] (B) edge (D);
\end{tikzpicture}
\end{center}
where the horizontal maps are given by Corollary \ref{cor:isom}. Therefore, applying Proposition \ref{prop:box}, we have
\[
H_0(\X^V)\xrightarrow{\cong} \sD(V)\overset{\text{w.h.p.}}{\cong} \im\left(\hat{\sD_1}(V)\rightarrow \hat{\sD_2}(V)\right). 
\]
Since $\hat{D}_{L_1+\varepsilon_1}^V\subset\hat{D}_{L_1-\varepsilon_1}^V\subset\hat{D}_{L_2+\varepsilon_2}^V\subset  \hat{D}_{L_2-\varepsilon_2}^V$, we have that $\im\left(\hat{\sD_1}(V)\rightarrow \hat{\sD_2}(V)\right)= \hat{\sD}(V)$. The map $\hat{\sD}(V)\rightarrow \hat{\sD}(W)$ in the statement of the proposition (and the commutativity of the resulting diagram) is induced by the inclusion $\hat{D}_{L_2-\varepsilon_2}^V\hookrightarrow \hat{D}_{L_2-\varepsilon_2}^W$ in the following commutative diagram. \qed

\begin{center}
\begin{tikzpicture}[scale=0.66, rotate=90]
\node (A) at (-2,1) {$\hat{D}^V_{L_1+\varepsilon_1}$};
\node (B) at (2,1) {$\hat{D}^V_{L_1-\varepsilon_1}$};
\node (C) at (0,4) {$D^V_{L_1}$};
\node (D) at (-2,-2) {$\hat{D}^W_{L_1+\varepsilon_1}$};
\node (E) at (2,-2) {$\hat{D}^W_{L_1-\varepsilon_1}$};
\node (F) at (0,-5) {$D_{L_1}^W$};
\node (G) at (-4.25,-3) {$D^W_{L_1+2\varepsilon_1}$};
\node (H) at (-4.25,2) {$D^V_{L_1+2\varepsilon_1}$};
\node (I) at (4.25,-3) {$D^W_{L_1-2\varepsilon_1}$};
\node (J) at (4.25,2) {$D^V_{L_1-2\varepsilon_1}$};

\path[right hook->] (A) edge (B);
\path[right hook->] (A) edge (D);
\path[right hook->] (A) edge (C);
\path[right hook->] (B) edge (E);
\path[right hook->] (C) edge (B);
\path[right hook->] (D) edge (E);
\path[right hook->] (D) edge (F);
\path[right hook->] (C) edge (F);
\path[right hook->] (F) edge (E);
\path[right hook->] (H) edge (A);
\path[right hook->] (B) edge (J);
\path[right hook->] (G) edge (D);
\path[right hook->] (E) edge (I);
\path[right hook->] (J) edge (I);
\path[right hook->] (H) edge (G);
\path[right hook->] (H) edge (C);
\path[right hook->] (C) edge (J);
\path[right hook->] (G) edge (F);
\path[right hook->] (F) edge (I);

\node (A2) at (9,1) {$\hat{D}^V_{L_2+\varepsilon_2}$};
\node (B2) at (13,1) {$\hat{D}^V_{L_2-\varepsilon_2}$};
\node (C2) at (11,4) {$D^V_{L_2}$};
\node (D2) at (9,-2) {$\hat{D}^W_{L_2+\varepsilon_2}$};
\node (E2) at (13,-2) {$\hat{D}^W_{L_2-\varepsilon_2}$};
\node (F2) at (11,-5) {$D_{L_2}^W$};
\node (G2) at (6.75,-3) {$D^W_{L_2+2\varepsilon_2}$};
\node (H2) at (6.75,2) {$D^V_{L_2+2\varepsilon_2}$};
\node (I2) at (15.25,-3) {$D^W_{L_2-2\varepsilon_2}$};
\node (J2) at (15.25,2) {$D^V_{L_2-2\varepsilon_2}$};

\path[right hook->] (A2) edge (B2);
\path[right hook->] (A2) edge (D2);
\path[right hook->] (A2) edge (C2);
\path[right hook->] (B2) edge (E2);
\path[right hook->] (C2) edge (B2);
\path[right hook->] (D2) edge (E2);
\path[right hook->] (D2) edge (F2);
\path[right hook->] (C2) edge (F2);
\path[right hook->] (F2) edge (E2);
\path[right hook->] (H2) edge (A2);
\path[right hook->] (B2) edge (J2);
\path[right hook->] (G2) edge (D2);
\path[right hook->] (E2) edge (I2);
\path[right hook->] (J2) edge (I2);
\path[right hook->] (H2) edge (G2);
\path[right hook->] (H2) edge (C2);
\path[right hook->] (C2) edge (J2);
\path[right hook->] (G2) edge (F2);
\path[right hook->] (F2) edge (I2);

\node (A3) at (-5.5,5) {$\X^V$};
\node (B3) at (-5.5,-6) {$\X^W$};
\node (C3) at (5.5,7) {$\X^V_{\delta_1}$};
\node (D3) at (5.5,-8) {$\X^W_{\delta_1}$};
\node (E3) at (16.5,5) {$\X^V_{\delta_2}$};
\node (F3) at (16.5,-6) {$\X^W_{\delta_2}$};

\path[right hook->] (A3) edge (C);
\path[right hook->] (A3) edge (B3);
\path[right hook->] (B3) edge (F);
\path[right hook->] (F) edge (D3);
\path[right hook->] (C) edge (C3);
\path[right hook->] (C) edge (C2);
\path[right hook->] (F) edge (F2);
\path[right hook->] (C3) edge (C2);
\path[right hook->] (D3) edge (F2);
\path[right hook->] (C2) edge (E3);
\path[right hook->] (F2) edge (F3);
\path[right hook->] (E3) edge (F3);
\path[right hook->] (A3) edge (C3);
\path[right hook->] (C3) edge (E3);
\path[right hook->] (B3) edge (D3);
\path[right hook->] (D3) edge (F3);
\path[right hook->] (J) edge (H2);
\path[right hook->] (I) edge (G2);
\end{tikzpicture}
\end{center}

\end{proof}

\section{Main results}
\label{sec:MainResults}
In this section, we prove convergence of the random enhanced mapper graph to the Reeb graph, as well as stability of the enhanced mapper graph under certain perturbations of the corresponding real valued function. Using the model described in Section \ref{sec:Model}, we generate random data, which is used to define a cosheaf which estimates the connected components of fibers of the real valued function associated to a given constructible $\R$-space. In the proof of Theorem \ref{thm:main}, we show that, with high probability, the cosheaf constructed using random data is isomorphic to the mapper functor applied to the Reeb cosheaf defined in Section \ref{sec:Background}. We then use the results established in Section \ref{sec:Background} to translate the cosheaf theoretic statement into a geometric statement (Corollary \ref{cor:main}) for the corresponding $\R$-graphs. 

To begin, we identify a sufficient condition for determining when a morphism of constructible cosheaves is an isomorphism. 
A morphism $\sF\rightarrow\sG$ of cosheaves is a family of maps $\sF(V)\rightarrow \sG(V)$, for each open set $V$, which form a commutative diagram
\begin{center}
\begin{tikzpicture}[scale=1]
\node at (0,0) {$\circlearrowright$};
\node (A) at (-1.5,1) {$\sF(V)$};
\node (B) at (1.5,1) {$ \sG(V)$};
\draw[<-] (B) edge node[above] {} (A);
\node (C) at (-1.5,-1) {$\sF(W)$};
\node (D) at (1.5,-1) {$\sG(W)$};
\path[<-] (D) edge node[above] {} (C);
\path[->] (A) edge (C);
\path[->] (B) edge (D);
\end{tikzpicture}
\end{center}
for each pair of open sets $V\subset W$. The morphism $\sF\rightarrow \sG$ is an isomorphism if each of the maps $\sF(V)\rightarrow \sG(V)$ is an isomorphism. Our first result shows that for cosheaves of the form $\cM_\cU(\sF)$, it is sufficient to consider only the maps $\cM_\cU(\sF)(V)\rightarrow \cM_\cU(\sG)(V)$ for open sets $V\in\cU'$.
\begin{proposition}\label{prop:cosheaf-iso}
Let $\sC$ and $\sD$ be cosheaves on $\R$. An isomorphism $\cM_\cU(\sC)\rightarrow\cM_\cU(\sD)$ of cosheaves is uniquely determined by a family of isomorphisms $\cM_\cU(\sC)(V)\rightarrow \cM_\cU(\sD)(V)$ for each $V\in\cU'$, which form a commutative diagram
\begin{center}
\begin{tikzpicture}[scale=1]
\node at (0,0) {$\circlearrowright$};
\node (A) at (-2,1) {$\cM_\cU(\sC)(V)$};
\node (B) at (2,1) {$ \cM_\cU(\sD)(V)$};
\draw[<-] (B) edge node[above] {$\cong$} (A);
\node (C) at (-2,-1) {$\cM_\cU(\sC)(W)$};
\node (D) at (2,-1) {$\cM_\cU(\sD)(W)$};
\path[<-] (D) edge node[above] {$\cong$} (C);
\path[->] (A) edge (C);
\path[->] (B) edge (D);
\end{tikzpicture}
\end{center}
for each pair $V\subset W\in\cU'$.
\end{proposition}
\begin{proof}
Proposition \ref{prop:mapper-functor} shows that $\cM_\cU(\sC)$ and $\cM_\cU(\sD)$ are constructible cosheaves over $\R$. The proof then follows from \cite[Proposition 3.10]{deSilvaMunchPatel2016}.
\qed
\end{proof}
Recalling the notation of Section \ref{sec:Background} and Section \ref{sec:Model}, for a super-level set $D_L$ of $p$, let $\sR_{D_{L}}$ be the Reeb cosheaf of $(D_L,f)$ on $\R$, defined by
$$
\sR_{D_L}(U)=\pi_0(D_L^U)
$$
for each open set $U\subset \R$. Let $\sR_{\hat{D}_L}$ be the Reeb cosheaf of $(\hat{D}_L,f)$ on $\R$, defined by
$$
\sR_{\hat{D}_L}(U)= \pi_0(\hat{D}_L^U)
$$
where $\hat D_L, \hat D_L^U$ are defined in \eqref{eqn:def_DL}, \eqref{eqn:def_DLV}, respectively, and $U\subset \R$ is an open set. We should note that $(D_L,f)$ and $(\hat{D}_L,f)$ are not apriori constructible spaces, so the cosheaves $\sR_{D_L}$ and $\sR_{\hat{D}_L}$ are not necessarily constructible. However, in what follows we will work exclusively with $\cM_\cU(\sR_{D_L})$ and $\cM_\cU(\sR_{\hat{D}_L})$, which are constructible cosheaves by Proposition \ref{prop:mapper-functor}. 

   Let $\hat{\sD}^\pi_n$ be the cosheaf defined by
   \[
   \hat{\sD}_n^\pi:=\cM_\cU\left(\im\left(\sR_{\hat{D}_{L_1+\varepsilon_1}}\rightarrow \sR_{\hat{D}_{L_2-\varepsilon_2}}\right)\right),
   \]
   with constants $n$, $L_1$, $L_2$, $\varepsilon_1$, and $\varepsilon_2$ chosen in Section \ref{sec:Model}. More explicitly, $\hat{\sD}^\pi_n$ maps an open interval $U$ to elements of the set $\sR_{\hat{D}_{L_2-\varepsilon_2}}(\cI_\cU(U))$ which lie in the image of the set $\sR_{\hat{D}_{L_1+\varepsilon_1}}(\cI_\cU(U))$ under the map induced by the inclusion $\hat{D}_{L_1+\varepsilon_1}\subseteq\hat{D}_{L_2-\varepsilon_2}$. By Proposition \ref{prop:mapper-functor}, $\hat{\sD}^\pi_n$ is a constructible cosheaf. 
\begin{theorem}\label{thm:main}
Assume there exists $\epsilon<\delta_\cU$ such that $p$ is $\epsilon$-concentrated on $\X$, then
    \[  \limninf\prob{d_I(\hat{\sD}^\pi_n,\sR_\X)\le\text{res}_f\cU} =1.\]
\end{theorem}

 \begin{proof}
\myblue{An inclusion of open sets $Y\subset Z$ induces a map 
\[
\pi_0(Y)\rightarrow \pi_0(Z),
\]
of the corresponding sets of path-connected components of $Y$ and $Z$ respectively. Each set of path-connected components forms a basis for the homology group in degree 0. Therefore, the map from $\pi_0(Y)$ to $\pi_0(Z)$ extends to a map between homology groups, resulting in the following commutative diagram
\begin{center}
\begin{tikzpicture}[scale=1]
\node at (0,0) {$\circlearrowright$};
\node (A) at (-1,1) {$\pi_0(Y)$};
\node (B) at (1,1) {$H_0(Y) $};
\path[<-] (B) edge node {} (A);
\node (C) at (-1,-1) {$\pi_0(Z)$};
\node (D) at (1,-1) {$H_0(Z).$};
\path[<-] (D) edge node {} (C);
\path[->] (A) edge (C);
\path[->] (B) edge (D);
\end{tikzpicture}
\end{center}
By combining the preceding commutative diagram with Proposition \ref{prop:level2-box}, we see that for every $V\subset W\in \cU'$, the following diagram commutes w.h.p. }
    \begin{center}
\begin{tikzpicture}[scale=1]
\node at (0,0) {$\circlearrowright$};
\node (A) at (-1,1) {$\pi_0(\X^{V})$};
\node (B) at (1,1) {$\hat{\sD}^\pi_n(V) $};
\path[<-] (B) edge node[above] {$\cong$} (A);
\node (C) at (-1,-1) {$\pi_0(\X^W)$};
\node (D) at (1,-1) {$\hat{\sD}^\pi_n(W).$};
\path[<-] (D) edge node[above] {$\cong$} (C);
\path[->] (A) edge (C);
\path[->] (B) edge (D);
\end{tikzpicture}
\end{center}
Notice that if $V\in\cU'$, then $\cI_\cU(V)=V$. By Proposition \ref{prop:cosheaf-iso} we have that $$\hat{\sD}^\pi_n\overset{\text{w.h.p.}}{\cong}\cM_\cU(\sR_\X).$$ Therefore, w.h.p.
\[
d_I(\hat{\sD}^\pi_n,\cM_\cU(\sR_\X))=0.
\]
Theorem \ref{thm:interleave}, combined with the triangle inequality, implies the theorem.
\qed
 \end{proof}

 \begin{corollary}\label{cor:main}
 Let $\cR(\X,f)$ be the Reeb graph of a constructible $\R$-space $(\X,f)$, and $\cD(\hat{\sD}^\pi_n)$ be the display locale of the mapper cosheaf defined above. If there exists $\epsilon<\delta_\cU$ such that $p$ is $\epsilon$-concentrated on $\X$, then
  \[  \limninf\prob{d_R\big(\cD(\hat{\sD}^\pi_n),\cR(\X,f)\big)\le\text{res}_f\cU} =1.\]
 \end{corollary}
 
 If $p$ is concentrated on $\X$, then the above corollary will hold for nice open covers with arbitrarily small resolution, as long as $\delta_\cU$ remains positive. Therefore, Corollary~\ref{cor:main} implies that we can use random point samples from $p$ to construct mapper graphs that are (w.h.p.) arbitrarily close (in the Reeb distance) to the Reeb graph of $\X$. 

To conclude, we will turn our attention to the stability of mapper cosheaves corresponding to a constructible space $(\X,f)$ under perturbations of the function $f$. The following theorem uses the machinery of cosheaf theory to prove that the mapper cosheaf is stable as long as the singular points of the constructible $\R$-space $\X$ are sufficiently ``far away" from the set of boundary points of our open cover $\cU$. 
 \begin{theorem}\label{thm:stability}
     Suppose $\sF$ and $\sG$ are constructible cosheaves over $\R$, with a common set of critical values $S$. Let $\cU$ be a nice open cover of $\R$, with set of boundary points $B$. Assume that
     \[
     d_I(\sF,\sG)<\min\{|s-b|:s\in S,b\in B\}.
     \]
     Then
     \[
     d_I(\cM_\cU(\sF),\cM_\cU(\sG))<d_I(\sF,\sG).
     \]
     Moreover, if $\sF$ is the Reeb cosheaf of $(\X,f)$ and $\sG$ is the Reeb cosheaf of $(\X,g)$, then
     \[
     d_I(\cM_\cU(\sF),\cM_\cU(\sG))<||f-g||_\infty.
     \]
 \end{theorem}
 \begin{proof}
 Suppose $\varphi_U:\sF(U)\rightarrow \sG(U_\varepsilon)$ and $\psi_U:\sG(U)\rightarrow\sF(U_\varepsilon)$ give an $\varepsilon$-interleaving of $\sF$ and $\sG$. Recall that
 \[
 \cM_\cU(\sF)(U)=\sF(\cI_\cU(U)).
 \]
 Then
 \[
 \varphi_{\cI_\cU(U)}:\cM_\cU(\sF)(U)\rightarrow \sG(\cI_\cU(U)_\varepsilon).
 \]
 In general, this does not give us an $\varepsilon$-interleaving of $\cM_\cU(\sF)$ and $\cM_\cU(\sG)$, because $\cI_\cU(U)_\varepsilon\neq \cI_\cU(U_\varepsilon)$. However, we will proceed by showing that each of these sets contain the same set of critical values.
 
Following the definition of $\cI_\cU$, we see that for each $U\in \Int$, $\cI_\cU(U)$ is an open interval in $\R$, with boundary points contained in $B$. Therefore $\cI_\cU(U)\cap S\subset \cI_\cU(U)_\epsilon\cap S$. If the inclusion is not an equality, then there must exist $s\in S$ such that $s\in \cI_\cU(U)_\epsilon$ and $s\notin \cI_\cU(U)$. In other words, if $\cI_\cU(U)\cap S\subsetneq \cI_\cU(U)_\epsilon\cap S$, then there exists $s\in S$ and $b\in B$ such that $|s-b|<\epsilon$. 

Define
 \[
 N_{\cU,\varepsilon}(U):=\cI_\cU(U_\varepsilon)\cap \cI_\cU(U)_\varepsilon.
 \]
By the arguments above, if $ \varepsilon<\min\{|s-b|:s\in S,b\in B\}$, then $\cI_\cU(U)\cap S= \cI_\cU(U)_\epsilon\cap S$. It follows that $N_{\cU,\varepsilon}(U)\cap S = \cI_\cU(U)\cap S = \cI_\cU(U)_\varepsilon\cap S$, because $\cI_\cU(U)\subset \cI_\cU(U_\epsilon)$. By the definition of constructibility, this implies that the natural extension map $\sG[N_{\cU,\varepsilon}(U)\subset \cI_\cU(U)_\varepsilon]$ (denoted by $e$ for notational brevity)  
 \[
 \sG(N_{\cU,\varepsilon}(U))\xrightarrow{\quad e \quad} \sG(\cI_\cU(U)_\varepsilon)
 \]
is an isomorphism, and therefore is invertible. The composition
 \[
 \cM_\cU(\sF)(U)\xrightarrow{\varphi}\sG(\cI_\cU(U)_\varepsilon)\xrightarrow{e^{-1}}  \sG(N_{\cU,\varepsilon}(U))\rightarrow \sG(\cI_\cU(U_\varepsilon))= \cM_\cU(\sG)(U_\varepsilon)
 \]
 gives an $\varepsilon$-interleaving of $\cM_\cU(\sF)$ and $\cM_\cU(\sG)$, because each map in the composition is natural with respect to inclusions. Therefore 
 \[
     d_I(\cM_\cU(\sF),\cM_\cU(\sG))<d_I(\sF,\sG).
 \]
When $\sF$ is the Reeb cosheaf of $(\X,f)$ and $\sG$ is the Reeb cosheaf of $(\X,g)$, the second statement of the theorem is a direct consequence of the above inequality and \cite[Theorem 4.4]{deSilvaMunchPatel2016}. 
\qed
 \end{proof}


\section{Discussion}
\label{sec:discussion}

In this paper, we work with a categorification of the Reeb graph~\citep{deSilvaMunchPatel2016} and introduce a categorified version of the mapper construction. 
This categorification provides the framework for using cosheaf theory and interleaving distances to study convergence and stability for mapper constructions applied to point cloud data. In this setting, the Reeb graph of a constructible $\R$-space is realized as the display locale of a constructible cosheaf (which we refer to as the Reeb cosheaf, following \cite{deSilvaMunchPatel2016}). In Section \ref{sec:CategorifiedMapper}, we define a mapper functor from the category of cosheaves to the category of constructible cosheaves, giving a category theoretic interpretation of the mapper construction. We then define the \emph{enhanced mapper graph} to be the display locale of the mapper functor applied to the Reeb cosheaf. 
We give an explicit geometric realization of the display locale as the quotient of a disjoint union of closed intervals, as illustrated in Figure \ref{fig:discussion}.  
In Section \ref{sec:Model}, we give a model for randomly sampling points from a probability density function concentrated on a constructible $\R$-space. After applying kernel density estimates, we consider an enhanced mapper graph generated by the random data. The main result of the paper, Theorem \ref{thm:main}, then gives (with high probability) a bound on the Reeb distance between the Reeb graph and the enhanced mapper graph generated by a random sample of points. 

\begin{figure}[!ht]
\begin{center}
\includegraphics[width=0.80\textwidth]{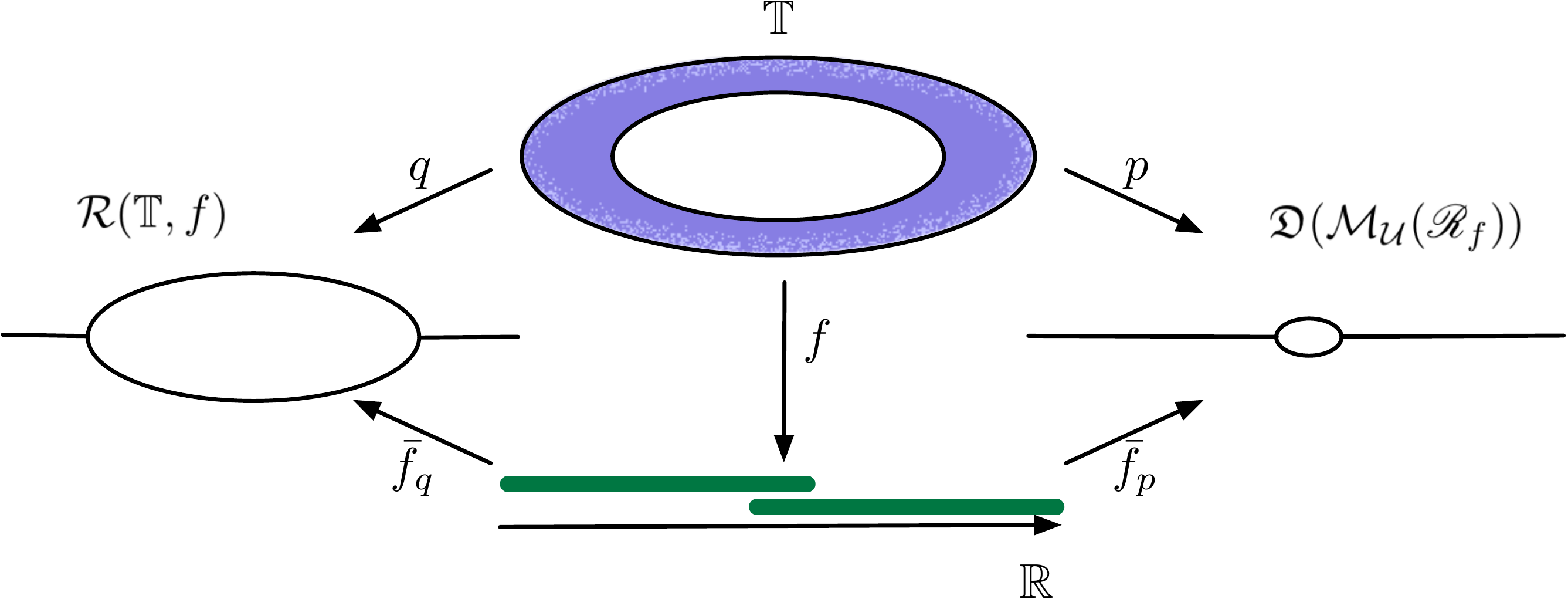} 
\caption{An example of the enhanced mapper graph $\mathfrak{D}(\cM_\cU(\sR_f))$ and the Reeb graph $\cR(\mathbb{T},f)$ of the height function $f$ on the torus $\mathbb{T}$, with an open cover $\cU$ of $f(\mathbb{T})$ consisting of two open intervals. The maps $q$ and $\overline{f}_q$ are the natural quotient factorization of $f$ obtained from the definition of the Reeb graph. Similarly, $p$ and $\overline{f}_p$ are the quotient map and factorization of $f$ obtained from the definition of the enhanced mapper graph.}
\label{fig:discussion}
\end{center}
\end{figure}

\para{Refinement to classic mapper graph.}
The enhanced mapper graph suggests a few refinements to the classic mapper construction. Firstly, rather than an open cover $\cU$ of $f(\X)$ (the image of the constructible $\R$-space $\X$ in $\R$), it is more natural from the enhanced mapper perspective to start with a finite subset $\R_0$ of $\R$. From this finite subset, the enhanced mapper graph can be computed by first producing a finite disjoint union of closed intervals, with each interval associated to a connected component of the complement of $\R_0$. Then, by prescribing attaching maps on boundary points of the disjoint union of closed intervals, one can obtain a combinatorial description of the enhanced mapper graph as a graph with vertices labeled with real numbers. The enhanced mapper graph then has the structure of a stratified cover of $f(\X)$, the image of the constructible $\R$-space $\X$ in $\R$. As such, the enhanced mapper graph contains more information than the classic mapper graph. Specifically, edges of the enhanced mapper graph have a naturally defined length which captures geometric information about the underlying constructible $\R$-space. Therefore, the enhanced mapper graph is naturally geometric, meaning that it comes equipped with a map to $\R$. 


\para{Variations of mapper graphs.}
We return to an in-depth discussion among variations of classic mapper graphs. 
As illustrated in Figure~\ref{fig:torus-enhanced-mapper} for the $\Rspace$ $(\mathbb{T}, f)$, that is, a torus equipped with a height function,  the enhanced mapper graphs (g), geometric mapper graphs (i) studied by~\cite{MunchWang2016}, and multinerve mapper graphs (j), have all been shown to be interleaved with Reeb graphs (b)~\citep{MunchWang2016,CarriereOudot2018}. To further illustrate the subtle differences among the enhanced, geometric, mutinerve and classic mapper graphs, we give additional examples in Figure~\ref{fig:enhanced-mapper-revisited} and Figure~\ref{fig:segments-enhanced-mapper}. 
In certain scenarios, some of these constructions appear to be identical or very similar to each other. 
We would like to understand the information content associated with the above variants of mapper graphs, all of which are used as approximations of the Reeb graph of a constructible $\R$-space. 
As illustrated in Figure~\ref{fig:torus-enhanced-mapper}, given an enhanced mapper graph (g) and an open cover (c), one can recover the the multinerve mapper graph (j), the geometric mapper graph (i), and the classic mapper graph (k). 
In future work, it would be interesting to quantify precisely the reconstruction ordering of these variants with and without any knowledge of the open cover.

In order to study convergence and stability of each variation of the mapper graph, it is necessary to assign function values to vertices of the graph. For the classic mapper graph or multinerve mapper graph, each vertex can be assigned, for instance, the value of the midpoint of a corresponding interval in $\R$. However, the display locale of a cosheaf over $\R$ admits a natural projection onto the real line, making a choice of function values unnecessary for the enhanced mapper graph. For this reason, we view the enhanced mapper graph as a natural variation of the mapper graph, well-suited for studying stability and convergence, with a natural interpretation in terms of cosheaf theory. 

\begin{figure}[!ht]
\begin{center}
\includegraphics[width=0.80\textwidth]{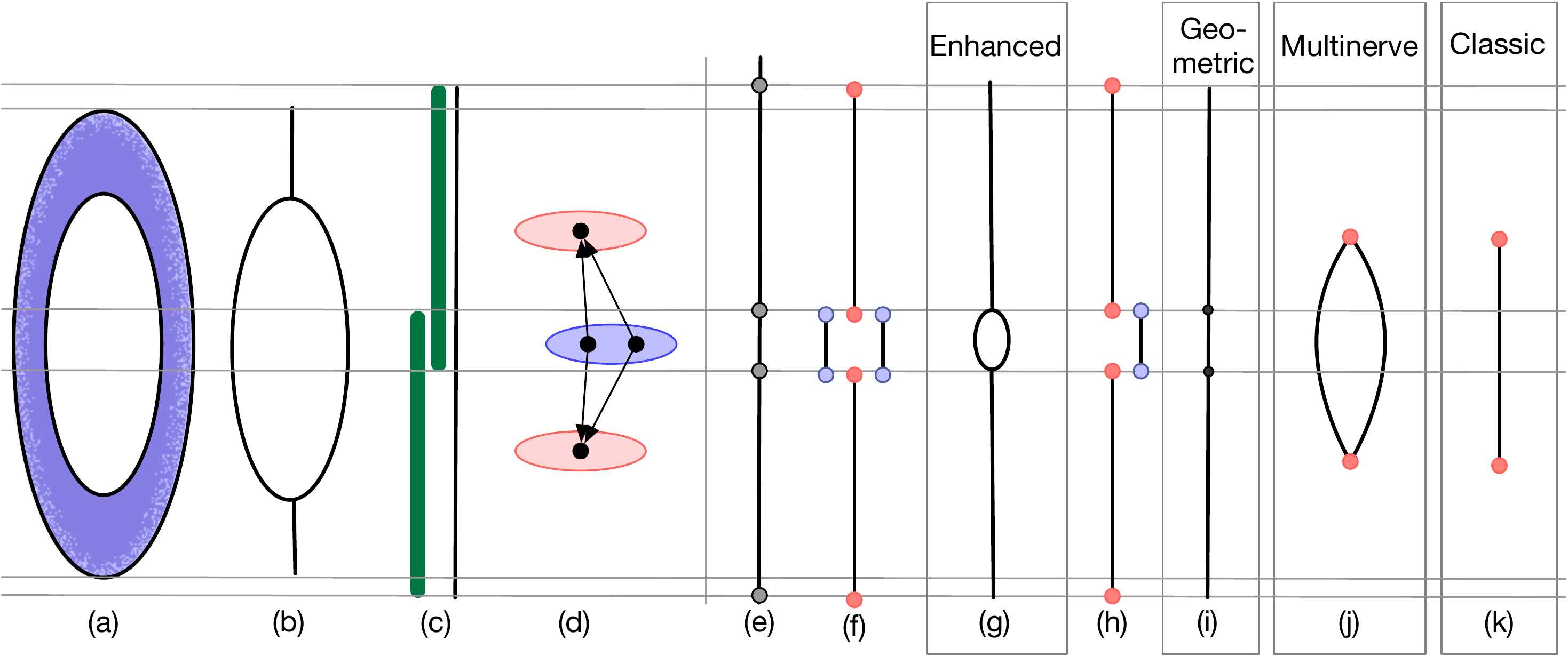} 
\caption{Variations of mapper graphs for the height function on a torus. (a) Torus with a height function. (b) Reeb graph. (c) Nice cover. (d) Visualization of the mapper cosheaf. (e) Stratification of $\R$. (f) Disjoint union of closed intervals, $\widetilde{\mathfrak{D}}(\cM_\cU(\sR_f))$, with quotient isomorphic to the enhanced mapper graph. (g) Enhanced mapper graph, $\mathfrak{D}(\cM_\cU(\sR_f))$. 
     (h) Disjoint union of closed intervals used to construct geometric mapper graph~\citep{MunchWang2016}. 
     (i) Geometric mapper graph. (j) Multinerve mapper graph. (k) Classic mapper graph.}
\label{fig:torus-enhanced-mapper}
\end{center}
\end{figure}

\begin{figure}[!ht]
\begin{center}
\includegraphics[width=0.85\textwidth]{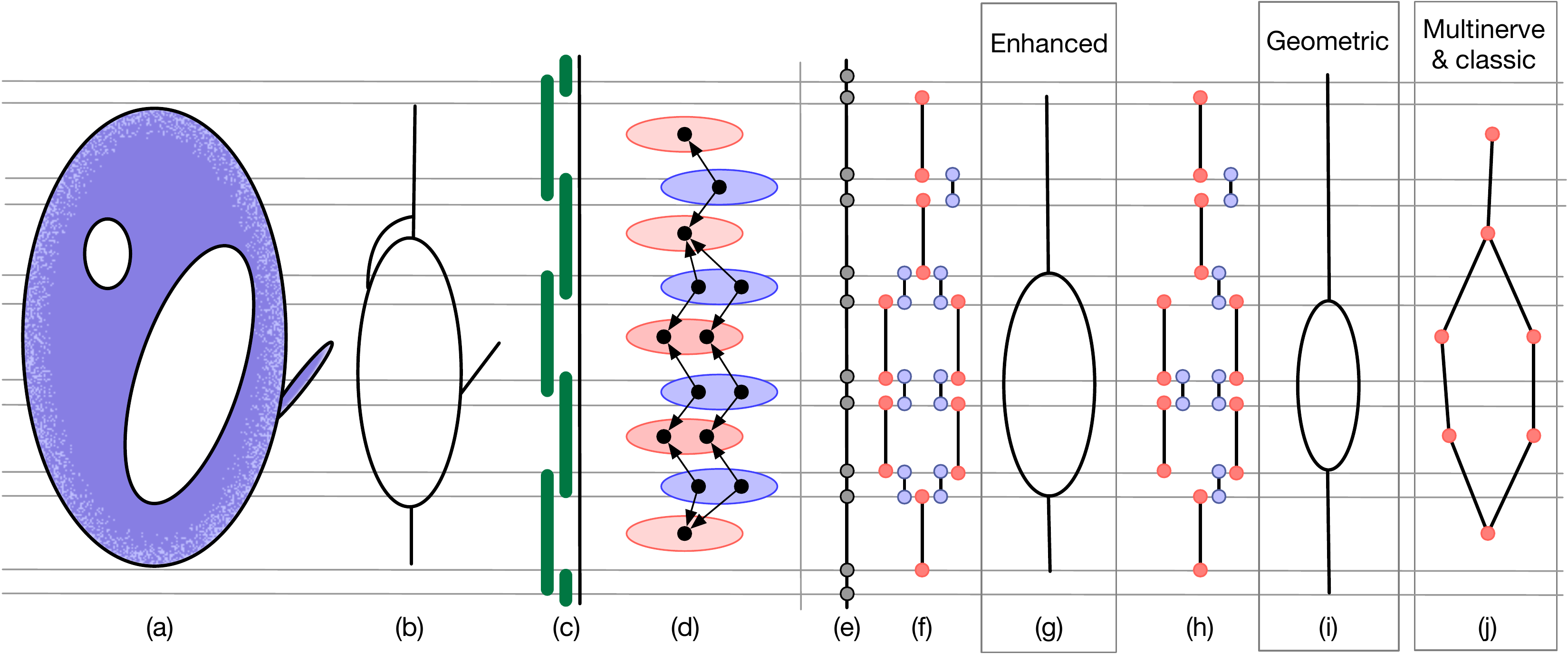} 
\caption{A return to the example illustrated in Figure~\ref{fig:enhanced-mapper}. Variations of mapper graphs of a height function on a topological space. (a) A topological space with a height function. (b) Reeb graph. (c) Nice cover. (d) Visualization of the mapper cosheaf. (e) Stratification of $\R$. (f) Disjoint union of closed intervals with quotient isomorphic to the enhanced mapper graph. (g) Enhanced mapper graph. 
     (h) Disjoint union of closed intervals used to construct geometric mapper graph~\citep{MunchWang2016}. 
     (i) Geometric mapper graph. (j) Multinerve and classic mapper graph.}
\label{fig:enhanced-mapper-revisited}
\end{center}
\end{figure}

\begin{figure}[!ht]
\begin{center}
\includegraphics[width=0.75\textwidth]{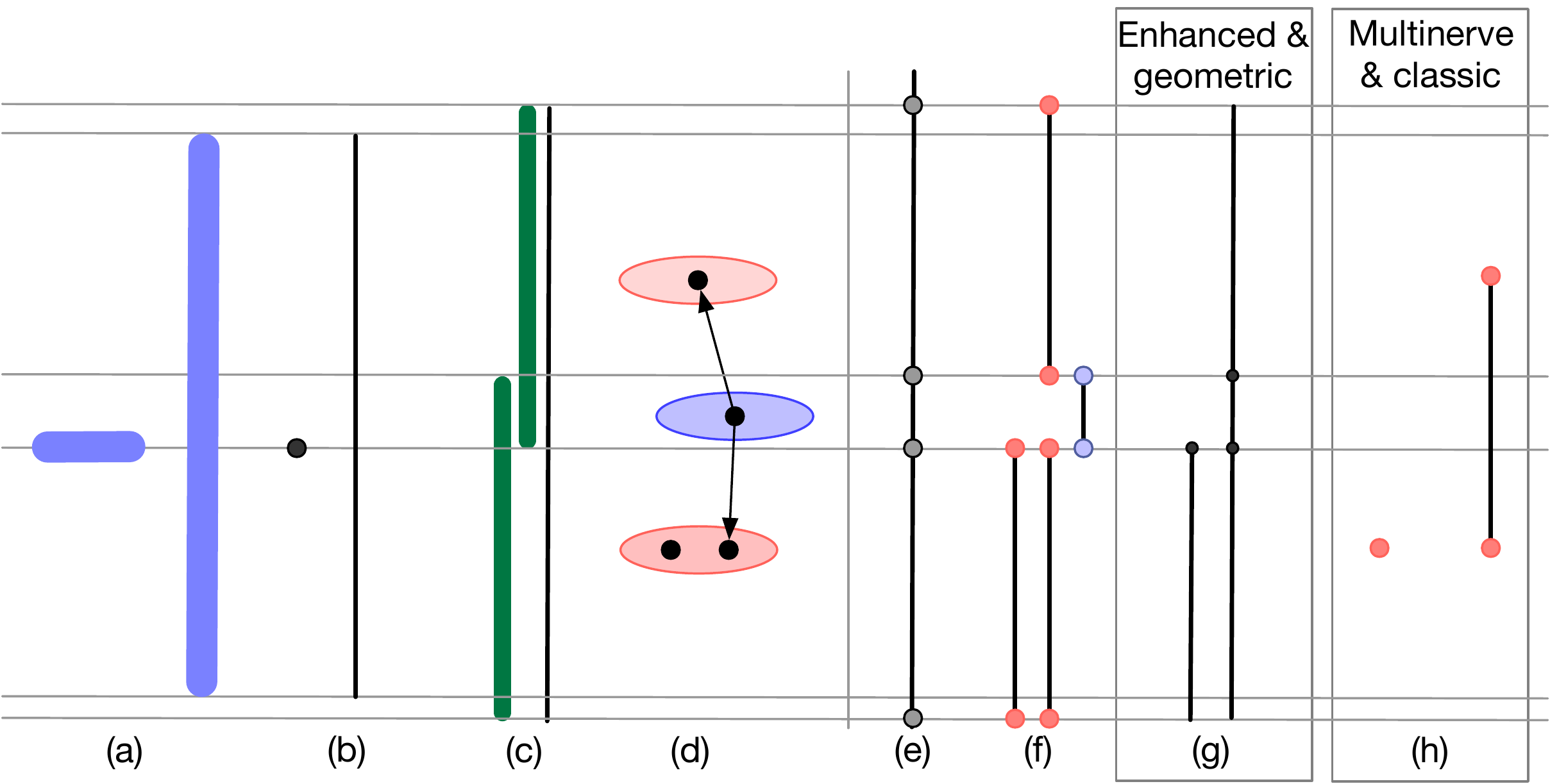} 
\caption{Variations of mapper graphs of a height function on a topological space consisting of two line segments. (a) A topological space consisting of two line segments. (b) Reeb graph. (c) Nice cover. (d) Visualization of the mapper cosheaf. (e) Stratification of $\R$. (f) Disjoint union of closed intervals with quotient isomorphic to the enhanced mapper graph. (g) Enhanced and geometric mapper graph. 
     (i) Multinerve and classic mapper graph.}
\label{fig:segments-enhanced-mapper}
\end{center}
\end{figure}

\para{Multidimensional setting and parameter tuning.}
 It is natural to extend the enhanced mapper graph (and more generally the categorification of mapper graphs) to multidimensional Reeb spaces and multi-parameter mapper through studying constructible cosheaves and stratified covers of $\R^N$, for $N>1$. We would also like to study the behavior of the parameter $\delta_\cU$ for various constructible spaces and open covers. In general, this parameter can vanish for ``bad" choices of open cover $\cU$. It would be worthwhile to extend the results of this paper to obtain bounds on the interleaving distance when $\delta_\cU$ vanishes. In conclusion, we hope for the results of this paper to promote the utility of combining methods from statistics and sheaf theory for the purpose of analyzing algorithms in computational topology. 

\begin{acknowledgements} 
AB was supported in part by the European Union’s Horizon 2020 research and innovation programme under the Marie Sklodowska-Curie Grant Agreement No. 754411 and NSF IIS-1513616.
OB was supported in part by the Israel Science Foundation, Grant 1965/19.
BW was supported in part by NSF IIS-1513616 and DBI-1661375.
EM was supported in part by NSF CMMI-1800466, DMS-1800446, and  CCF-1907591. 
We would like to thank the Institute for Mathematics and its Applications for hosting a workshop titled \emph{Bridging Statistics and Sheaves} in May 2018, where this work was conceived.
\end{acknowledgements}

\section*{Conflict of interest}
The authors declare that they have no conflict of interest.

\appendix
\section{\myblue{Pseudocode for the Enhanced Mapper Graph Algorithm}}
\label{sec:pseudocode}

The following pseudocode (Algorithm~\ref{algorithm:enhanced-mapper}) outlines an algorithm for computing the enhanced mapper graph, which is stored as a graph $G=(F,E)$ with a vertex set $F$ and an edge set $E$, together with a real-valued function $f:F\rightarrow \R$.

The algorithm assumes that we are given sets $\pi_0(f^{-1}(U))$ (denoted by $\Sigma$ in the pseudocode) and set maps $\pi_0(f^{-1}(U))\rightarrow \pi_0(f^{-1}(V))$ (denoted by $\rho$ in the pseudocode) for various $U \subset V\subset \R$. 
In other words, the algorithm assumes that there is an oracle (referred to as a \emph{set oracle}) that takes as input an inverse mapping of an interval and returns its corresponding set of path-connected components. 
It also assumes that there is a \emph{set-map oracle} that keeps tracks of set maps between a pair of path-connected components (each component is denoted by $s$ in the pseudocode). 
In Section \ref{sec:Model}, we give a statistical approach for computing such sets and set maps through kernel density estimates. 

In Algorithm~\ref{algorithm:enhanced-mapper}, let $\mathcal{U} = \{U_{i}\}_{i \in A}$ denote a finite set of pairwise intersecting open intervals.
For simplicity, suppose the index set $A \subset \mathbb{Z}$ contains consecutive integers.  
That is, for each interval $U_{i}:=(u_{i}^-,u_{i}^+)$ (for some $i \in A$), we have ~$u_{i}^-<u_{i-1}^+<u_{i+1}^-<u_i^+<u_{i+1}^+$ (assuming $i-1, i+1 \in A$). 
For each interval $U_i$, $\Sigma_i := \pi_0(f^{-1}(U_i))$ denotes the set of path-connected components. 
For each path-connected component $s \in \Sigma_i$, the pairs $(s,+)\in \Sigma_i\times\{+,-\}$ and $(s,-)\in \Sigma_i\times\{+,-\}$ represent the two vertices associated to the edge in the enhanced mapper graph which corresponds to $s$. 
Similarly, for each path-connected component $t \in \Sigma_{(i,i+1)}$, the pairs $(\rho_i^-(t),+)$ and $(\rho_i^+(t),-)$ represent the two vertices associated to the edge  in the enhanced mapper graph which corresponds to $t$. 

For clarity, Figure~\ref{fig:pseudocode} illustrates notations used in the pseudocode of Algorithm~\ref{algorithm:enhanced-mapper}. 
It is based on a zoomed view of Figure~\ref{fig:enhanced-mapper}(c)-(f). 
The maps $\rho_i^-$ and $\rho_i^+$ define how the red vertices and blue vertices (as end points of intervals) are glued together to form an enhanced mapper graph. 
In this particular example,  $(\rho^-_i(t), +)$ (a blue vertex) matches with $(s,+)$ (a red vertex), due to the fact that $\rho_i^-(t)=s$.


\begin{figure}[!ht]
\begin{center}
\includegraphics[width=0.85\textwidth]{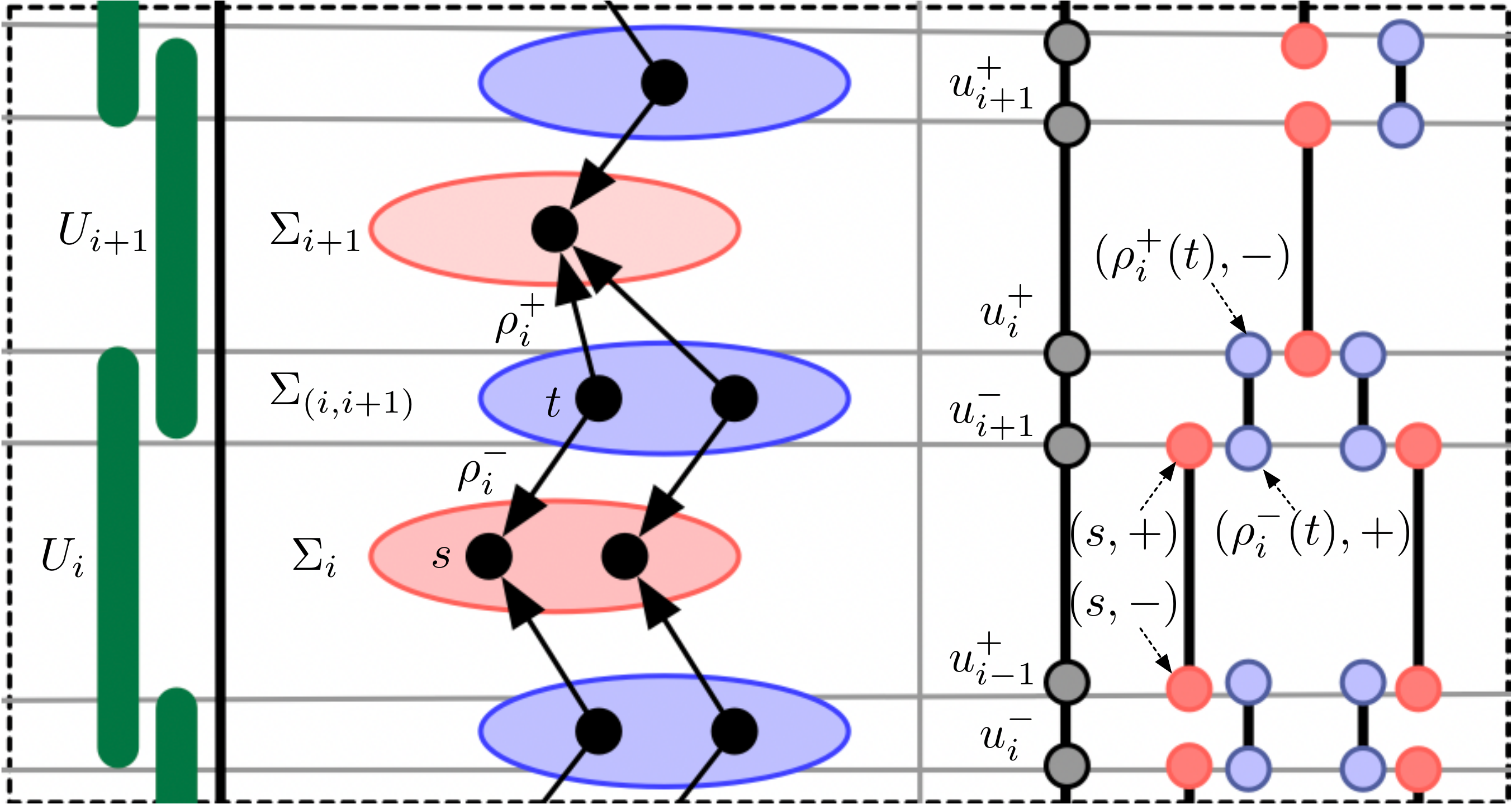} 
\caption{An illustration of notations used in the pseudocode of Algorithm~\ref{algorithm:enhanced-mapper}.}
\label{fig:pseudocode}
\end{center}
\end{figure}

\begin{algorithm}[H]
\label{algorithm:enhanced-mapper}
\caption{Compute an enhanced mapper graph with oracles. }

\Input{ 
\begin{itemize}
\item {A finite set of pairwise intersecting open intervals}: $\{U_{i}:=(u_{i}^-,u_{i}^+)\}_{i \in A}$ 
\item {For each interval, a set returned by a set oracle}: 
\begin{itemize}
\item For each $U_i$, a set  $\Sigma_i := \pi_0(f^{-1}(U_i))$  
\item For each $(U_i, U_{i+1})$, a set $\Sigma_{(i,i+1)}:=\pi_0(f^{-1}(U_i \cap U_{i+1}))$
\end{itemize}
\item {For each pair of intervals, a set map returned by a  set-map oracle}:\\
For each $(U_i, U_{i+1})$, set maps   
$$
\rho^-_{i}:\Sigma_{(i,i+1)}\rightarrow \Sigma_i, 
 \quad  \rho^+_{i}: \Sigma_{(i,i+1)}\rightarrow \Sigma_{i+1}.
$$ 
\end{itemize}} 

\Output{ 
\begin{itemize}
\item A graph $G=(F,E)$ with a vertex set $F$ and an edge set $E \subseteq F \times F$ 
\item A function $f:F\rightarrow \R$
\end{itemize}
}

Initialize $F=\emptyset$ and $E=\emptyset$ \\

\For{$i\in A$}{
Set $\Sigma^+_i:=\Sigma_i\times\{+\}$ \\
Set $\Sigma^-_i:=\Sigma_i\times\{-\}$\\
$F\leftarrow F\sqcup \Sigma^-_i\sqcup \Sigma^+_i$\\
	\For {$s\in\Sigma_i$}{
	$E\leftarrow E\sqcup ((s,-),(s,+))$
	\begin{eqnarray*}
	    f((s,+))&:=\begin{cases}
		    u_{i+1}^-&\text{ if }i+1\in  A\\
		    u_i^+&\text{ if }i+1\notin A,
		    \end{cases}\\
	    f((s,-))&:=\begin{cases}
		    u_{i-1}^+&\text{ if }i-1\in A\\
		    u_i^-&\text{ if }i-1\notin A,
		    \end{cases}
	\end{eqnarray*}}
}
\For {$(i,i+1) \in A \times A$}{
	\For {$t \in \Sigma_{(i,i+1)}$}{
		$E \leftarrow E \sqcup ((\rho_i^-(t),+),(\rho_i^+(t),-))$
	}
}
\end{algorithm}

\clearpage
\bibliography{mapper}

\end{document}